\documentclass[12pt,reqno]{amsart}
\usepackage[sorted,compressed-cites,sorted-cites,initials]{amsrefs}

\usepackage[all,cmtip]{xy}
\usepackage{enumerate,pifont,mathrsfs,geometry,manfnt}
\usepackage{amsfonts}
\usepackage{amsmath}\usepackage{amsthm}
\usepackage{amssymb}
\usepackage[active]{srcltx}

\usepackage{geometry}
\usepackage[colorlinks]{hyperref}
\usepackage[colorlinks]{hyperref}
\hypersetup{
bookmarksnumbered,
pdfstartview={FitH},
breaklinks=true,
linkcolor=blue,
urlcolor=blue,
citecolor=blue,
bookmarksdepth=2
}

\newtheorem*{cor}{Corollary}
\newtheorem*{lem}{Lemma}
\newtheorem*{thm}{Theorem}
\newtheorem*{prop}{Proposition}

\theoremstyle{definition}

\theoremstyle{definition}

\newtheorem*{rem}{Remark}

\newcommand{\ract}{\triangleleft}
\newcommand{\lact}{\triangleright}
\newcommand{\blact}{{\scriptstyle\blacktriangleright}}
\DeclareMathOperator{\Gr}{Gr}
\DeclareMathOperator{\Top}{\mathsf{top}}
\DeclareMathOperator{\Hom}{Hom}
\DeclareMathOperator{\Ext}{Ext}

\DeclareMathOperator{\Proj}{Proj}
\DeclareMathOperator{\Inj}{Inj}
\DeclareMathOperator{\im}{Im}
\DeclareMathOperator{\cop}{cop}

\numberwithin{equation}{section}

\DeclareMathOperator{\soc}{\mathsf{soc}} 
\DeclareMathOperator{\dsoc}{\mathsf{Soc}}
\DeclareMathOperator{\dtop}{\mathsf{Top}}

\usepackage{pifont,mathrsfs}
\newcommand{\cat}[1]{\mathscr{#1}}
\newcommand{\ds}[1]{\langle #1\rangle}

\begin{document}

\newcommand{\stanch}[2]{\hypertarget{#1:#2}{}}
\newcommand{\stref}[3]{\hyperlink{#2:#1}{#3~\ref{#1}}}
\newcommand{\lemanch}[1]{\stanch{lem}{#1}}
\newcommand{\thmanch}[1]{\stanch{thm}{#1}}
\newcommand{\propanch}[1]{\stanch{prop}{#1}}
\newcommand{\coranch}[1]{\stanch{cor}{#1}}
\newcommand{\remanch}[1]{\stanch{rem}{#1}}
\newcommand{\defanch}[1]{\stanch{def}{#1}}
\newcommand{\thmref}[1]{\stref{#1}{thm}{Theorem}}
\newcommand{\secref}[1]{\hyperref[#1]{Section~\ref{#1}}}
\newcommand{\subsref}[1]{\hyperref[#1]{\S\ref{#1}}}
\newcommand{\lemref}[1]{\stref{#1}{lem}{Lemma}}
\newcommand{\propref}[1]{\stref{#1}{prop}{Proposition}}
\newcommand{\corref}[1]{\stref{#1}{cor}{Corollary}}
\newcommand{\remref}[1]{\stref{#1}{ref}{Remark}}
\newcommand{\defref}[1]{\stref{#1}{def}{Definition}}
\newcommand{\id}{\operatorname{id}}
\newcommand{\sgn}{\operatorname{sgn}}
\newcommand{\wt}{\operatorname{wt}}
\newcommand{\tensor}{\otimes}
\newcommand{\nc}{\newcommand}
\newcommand{\twedge}{{\textstyle\bigwedge}}
\newcommand{\rnc}{\renewcommand}
\newcommand{\dist}{\operatorname{dist}}
\newcommand{\qbinom}[2]{\genfrac[]{0pt}0{#1}{#2}}
\nc{\cal}{\mathcal} \nc{\goth}{\mathfrak} \rnc{\bold}{\mathbf}
\renewcommand{\frak}{\mathfrak}
\newcommand{\supp}{\operatorname{supp}}
\newcommand{\ad}{\operatorname{ad}}
\newcommand{\Maj}{\operatorname{Maj}}
\newcommand{\Ht}{\operatorname{ht}}
\renewcommand{\Bbb}{\mathbb}
\newcommand{\ZZ}{\mathbb Z}
\nc\bomega{{\mbox{\boldmath $\omega$}}} \nc\bpsi{{\mbox{\boldmath
$\Psi$}}}
 \nc\balpha{{\mbox{\boldmath $\alpha$}}}
 \nc\bpi{{\mbox{\boldmath $\pi$}}}

\newcommand{\lie}[1]{\mathfrak{#1}}
\nc{\head}{\operatorname{head}}
\def\im{\operatorname{Im}}
\nc\gr{\operatorname{gr}}
\nc\kk{\Bbbk}

\newgeometry{margin=1.2in}

\title[Koszul duality for semidirect products]{Koszul duality for semidirect products and generalized Takiff algebras}

\author{Jacob Greenstein}
\address{Department of Mathematics, University of California Riverside, CA 92521, USA}
\email{jacob.greenstein@ucr.edu}

\author{Volodymyr Mazorchuk}
\address{Department of Mathematics, Uppsala University, Box 480, SE-751 06, Uppsala, Sweden}
\email{mazor@math.uu.se}

\begin{abstract}
We obtain Koszul-type dualities for categories of graded modules over
a graded associative algebra 
which can be realized as the semidirect product of a bialgebra coinciding with its degree zero part 
and a graded module algebra for the latter. In particular, this applies to 
graded representations of the universal enveloping algebra of the Takiff Lie algebra (or the truncated current algebra) and its (super)analogues, and also 
to semidirect products of quantum groups with braided symmetric and exterior module algebras in case the latter are flat deformations
of classical ones.
\end{abstract}
\maketitle

\vspace*{-1cm}
\section{Introduction}\label{sec1}

Koszul property, as defined in \cite{Pr}, plays an important role in modern representation and structural 
theory of graded associative algebras. It typically occurs in the following setting. 
Let $A$ be a $\ZZ$-graded associative algebra over a field~$\kk$ whose non-zero homogeneous components appear only in 
non-negative degrees and are finite dimensional and whose degree zero part $A_0$ is semisimple. Consider the 
category of locally finite dimensional graded $A$-modules, which can be regarded as a non-semisimple 
deformation of the semisimple category of finite dimensional $A_0$-modules. Koszulity of $A$ is then 
formulated via the requirement that simple graded $A$-modules have so-called {\em linear} projective 
resolutions. One consequence of  Koszulity  is a derived equivalence between the bounded derived 
category of finitely generated graded  $A$-modules and a similar category for the quadratic dual 
of~$A$ (see \cite{BGS}). This classical Koszul duality has numerous generalizations and extensions
(see \cites{Ke,MOS,Ma} and references therein).

However, it often happens that one needs to consider graded modules over a graded algebra whose degree zero 
part is not semisimple. A typical example is the current algebra $\lie g[t]:=\lie g\tensor \mathbb C[t]$ 
of a finite dimensional simple Lie algebra~$\lie g$ over $\mathbb C$ which is intimately connected to,
in particular, the quantum affine algebra corresponding to~$\lie g$. The universal enveloping algebra of $\lie g[t]$ is naturally 
$\ZZ$-graded, with the degree zero part being isomorphic to the enveloping algebra of~$\lie g$. The 
latter is very far from being semisimple. However, the category of finite dimensional $\lie g$-modules 
is semisimple. This allows one to associate a Koszul algebra with the full subcategory of the category 
of finitely generated graded locally finite dimensional $\lie g[t]$-modules whose objects are 
annihilated by the Lie ideal $\lie g\tensor t^2\mathbb C[t]$ of~$\lie g[t]$ (cf.~\cites{CG-1,CG-2}).
Alternatively, such modules can be regarded as modules over the Takiff Lie algebra $\lie g\rtimes \lie g$ (\cite{Tak})
which is naturally isomorphic to $\lie g\tensor\mathbb C[t]/(t^2)$. The interest in that category stems 
from the observation that $q=1$ limits of celebrated Kirillov-Reshetikhin modules over quantum affine algebras (\cite{KR})
and also of certain generalizations of Kirillov-Reshetikhin modules known as minimal affinizations (\cite{Ch}) are
its objects, provided that $\lie g$ is of a classical type. Once one has a 
Koszul algebra, its quadratic dual is also Koszul, and the natural question is whether one can find a 
Lie-theoretic background for  a representation theory of that quadratic dual.

In the present paper we answer that question and establish the most general framework in which such a question 
can be answered. Namely, it turns out that for the Takiff Lie algebra and its generalizations, the 
``Koszul graded dual'' is, morally, a certain Lie superalgebra which is isomorphic to our initial Lie 
algebra as a $\lie g$-module. This is a special case of the following setup. Suppose that we have a 
$\ZZ$-graded algebra $A$ such that $A_0$ is a bialgebra and $A$ is a semidirect product of $A_0$ 
with a $\ZZ$-graded right $A_0$-module algebra~$H$. We consider a category of  
$\ZZ$-graded left $A$-modules whose graded pieces are in a suitable ``underlying'' category~$\cat C$ of $A_0$-modules (for example,
a semisimple category of finite dimensional $A_0$-modules if it is available). 
Assuming that $H$ is Koszul, we establish (see~\thmref{KDS170} and its Corollary) a Koszul-type duality between that category and a category 
of $\ZZ$-graded left modules with graded pieces in~$\cat C$ over the algebra $A^\circledast$, which is the semidirect product of the quadratic 
dual of~$H$ with the bialgebra $A^{\cop}$ which coincides with~$A$ as an algebra and has the opposite comultiplication. 
The case of the Takiff algebra described above corresponds to taking 
$A_0$ to be the enveloping algebra of~$\lie g$, $H$ to be the symmetric algebra 
of the adjoint representation of~$\lie g$ and $\cat C$ to be the category of finite dimensional $\lie g$-modules.

The paper is organized as follows. In~\secref{GEN} we collected basic generalities on 
graded algebras and categories of graded modules. \secref{snew3} contains results pertaining to module algebras and semidirect products that 
are needed for our construction.
In~\secref{snew4} we establish Koszul duality in the setting of semidirect products of bialgebras with 
their module algebras. In particular, it contains the main results of the paper (\thmref{KDS170} and
\corref{KDS190}). Finally, \secref{snew5} provides examples illustrating applications of our main result.

\subsection*{Acknowledgments}
A major part of this research was done during the visit of the first author to Uppsala University
which was supported by the Faculty of Natural Sciences of Uppsala University. Support and hospitality
of Uppsala University are gratefully acknowledged. The paper was completed during the Representation
Theory program at the Institute Mittag-Leffler, whose support and hospitality are deeply appreciated.
The first author is partially supported by the Simons foundation collaboration grant no.~245735.
The second author is partially supported by the Swedish Research Council,
by Knut and Alice Wallenbergs Stiftelse and by the Royal Swedish Academy of Sciences.
We thank the referee for very helpful comments and especially for pointing out an
important inaccuracy in the original version.

\section{Generalities}\label{GEN}

\subsection{} All algebras considered in this paper are unital and over a 
fixed base field $\kk$. If $A$ is an algebra and $M$ is a left (respectively, right) 
$A$-module, then we denote the action of $a\in A$ on~$m\in M$ by $a\lact m$ 
(respectively $m\ract a$). Given a category~$\cat A$, we write $X\in \cat A$ 
as a shortcut for $X$ being an object in~$\cat A$. 

\subsection{} Let $\cat B$ be an additive, $\kk$-linear and idempotent split category.
The {\em projective abelianization} $\overline{\cat B}$ of $\cat B$ is the category
defined as follows:
\begin{itemize}
\item  objects  of $\overline{\cat B}$ are all diagrams of the form
\begin{equation}\label{eqnm1}
\xymatrix{X\ar[rr]^{\alpha}&&Y,}
\end{equation}
where $X,Y\in\cat B$ and $\alpha\in\Hom_{\cat B}(X,Y)$;
\item for objects $X\overset{\alpha}{\longrightarrow}Y$ and 
$X'\overset{\alpha'}{\longrightarrow}Y'$, the set 
$\Hom_{\overline{\cat B}}(X\overset{\alpha}{\longrightarrow}Y,
X'\overset{\alpha'}{\longrightarrow}Y')$ is the quotient of the vector space 
formed by all solid diagrams of the form
\begin{equation}\label{eqnm2}
\vcenter{\xymatrix{
X\ar[rr]^{\alpha}\ar[d]_{f}&&Y\ar[d]^g\ar@{.>}[lld]_{\theta}\\
X'\ar[rr]^{\alpha'}&&Y'
}}
\end{equation}
by the subspace generated by all such solid diagrams for which there exists a morphism
$\theta$, depicted by the dotted arrow, such that $\alpha'\theta=g$.
\end{itemize}

Dually, the {\em injective abelianization} $\underline{\cat B}$ of $\cat B$ is the
category whose objects are diagrams of the form \eqref{eqnm1} and whose corresponding
morphisms are given as quotients of the vector space formed by all solid diagrams of
the form \eqref{eqnm2} by the subspace generated by all such solid diagrams for which
there exists a morphism $\theta$ as depicted by the dotted arrow such that 
$\theta\alpha=f$.
% Assume that $\cat B$ is {\em strongly locally finite}, that is 
% \begin{itemize}
% \item $\Hom_{\cat B}(X,Y)$ is finite dimensional over~$\kk$ for any $X,Y\in\cat B$;
% \item for any $X\in\cat B$ there are finitely many 
% isomorphism classes of indecomposable $Y\in\cat B$ such that $\Hom_{\cat B}(X,Y)\neq 0$;
% \item for any $X\in\cat B$ there are  finitely many 
% isomorphism classes of indecomposable $Y\in\cat B$ such that $\Hom_{\cat B}(Y,X)\neq 0$.
% \end{itemize}
% In this case $\overline{\cat B}$ is equivalent to the category of finitely generated 
% $\cat B^{\op}$-modules and, under this equivalence, the category $\cat B$ naturally 
% becomes a full and dense subcategory of the category of finitely generated projective 
% $\cat B^{\op}$-modules. Similarly, $\underline{\cat B}$ is also equivalent to the 
% category of finitely cogenerated $\cat B^{\op}$-modules and, under this equivalence,  
% the category $\cat B$ naturally becomes a full and dense subcategory of the category 
% of finitely cogenerated injective $\cat B^{\op}$-modules. 
% Analogous statement holds if 
% we replace the finiteness assumptions on $\cat B$ by the appropriate assumption that it is 
% positively graded with finite dimensional graded components and consider the corresponding 
% categories of graded modules. 
We refer the reader to \cite{Fr} and \cite{MM}*{\S3.1} for more details
on abelianizations.

\subsection{}\label{GEN.5} Let $\cat A$ be an idempotent split exact category. 
Denote by $\Proj(\cat A)$ (respectively, $\Inj(\cat A)$) the strictly full subcategory 
of $\cat A$ consisting of projective (respectively, injective) objects of~$\cat A$. 

The category $\Gr\cat A$ of $\ZZ$-graded objects over~$\cat A$ is an additive category whose objects 
are $\mathsf X=\bigoplus_{i\in\ZZ}\mathsf X_i$, where $\mathsf X_i\in\cat A$, and 
\begin{displaymath}
\Hom_{\Gr\cat A}(\mathsf X,\mathsf Y):=
\prod_{i\in\ZZ} \Hom_{\cat A}(\mathsf X_i,\mathsf Y_i). 
\end{displaymath}
Given $r\in\ZZ$ and $\mathsf X\in\Gr\cat A$, we denote by $\mathsf X\ds{r}$ the object in $\Gr\cat A$ such that 
$(\mathsf X\ds r)_i=\mathsf X_{i+r}$, $i\in\ZZ$. This defines the degree shift endofunctor $\ds{r}$ of~$\Gr\cat A$. We say 
that $\mathsf X\in\Gr \cat A$ is {\em bounded above} (respectively, {\em below}) if $\mathsf X_i=0$ for $i\gg 0$ 
(respectively, $i\ll 0$). The full subcategories of bounded above (respectively, below)
objects in~$\Gr\cat A$ are denoted by $\Gr^-\cat A$ (respectively, $\Gr^+\cat A$).

For~$j\in\ZZ$, define a functor $\Pi_j:\Gr\cat A\to\cat A$ by $\Pi_j(\mathsf X)=\mathsf X_j$ 
for all $\mathsf X\in\Gr\cat A$ and $\Pi_j(f)=f_j$ for all $f=(f_i)_{i\in\ZZ}\in\Hom_{\Gr\cat A}(\mathsf X,\mathsf Y)$ and all 
$\mathsf X,\mathsf Y\in\Gr\cat A$. Clearly, $\Pi_j$ is exact. On the other hand, define a functor 
$\gr_j:\cat A\to \Gr\cat A$ by $\gr_j(X)_i=0$ if~$i\not=j$ and $\gr_j(X)_j=X$, for $X\in\cat A$, 
while $\gr_j(f)_j=f$, $\gr_j(f)_i=0$, $i\not=j$, for all $f\in\Hom_{\cat A}(X,Y)$ and $X,Y\in\cat A$. Then $\gr_j$ is a full exact embedding.
It is immediate that $\Pi_j\circ \gr_j$ is isomorphic to the identify functor on~$\cat A$ for all $j\in\ZZ$. 

We say that a non-zero object~$\mathsf X$ of~$\Gr\cat A$ is concentrated in degree~$i$ provided that 
$\Pi_j(\mathsf X)=0$ unless $j=i$. Clearly, if $\mathsf X$ is concentrated in degree~$j$ then 
$\mathsf X\ds r$ is concentrated in degree~$j-r$, $\Pi_{j+r}=\Pi_{j}\circ\ds{r}$ and 
$\gr_{j+r}=\ds{r}\circ\gr_j$, for all $j,r\in\ZZ$.

\subsection{}\label{GEN.new1}
A pair $(\mathsf X^\bullet,d_{\mathsf X})$ 
where $\mathsf X^\bullet=\bigoplus_{i\in \ZZ}\mathsf X^i$ with $\mathsf X^i\in\cat A$, 
and
$d_{\mathsf X}\in\prod_{j\in\ZZ}\Hom_{\cat A}( \mathsf X^j,\mathsf X^{j+1})$ satisfies
$d_{\mathsf X}\circ d_{\mathsf X}=0$ is called a {\em complex} over~$\cat A$; $d_{\mathsf X}$ is called the differential. 
The object~$\mathsf X^i$, $i\in\ZZ$ is called the component of $(\mathsf X^\bullet,d_{\mathsf X})$ of homological degree~$i$,
and an object $X$ in~$\cat A$ identifies with the complex with zero differential whose only non-zero 
component is $X$ in homological degree~$0$ (the trivial complex of~$X$).
A complex $(\mathsf X^\bullet,d_{\mathsf X})$ is said to be {\em bounded above} 
(respectively,  {\em bounded below}) if $\mathsf X^i=0$ for all $i\gg0$ 
(respectively, $i\ll0$). The homotopy category of bounded below (respectively, bounded above) 
complexes over~$\cat A$ is denoted $\mathcal K^+(\mathscr A)$ (respectively, $\mathcal K^-(\cat A)$). 

If $\mathscr A$ is abelian, 
we denote by $\mathcal D^+(\mathscr A)$ (respectively, 
$\mathcal D^-(\cat A)$) the corresponding bounded below (respectively, bounded above) derived category.

% If both $\Proj(\cat A)$ and $\Inj(\cat A)$ are strongly locally finite, 
% then $\mathcal K^+(\Inj(\cat A))$ (respectively, $\mathcal K^-(\Proj(\cat A))$ is equivalent 
% to the bounded below (respectively, bounded above) derived  category $\mathcal D^+(\underline{\Inj(\cat A)})$ 
% (respectively, $\mathcal D^-(\overline{\Proj(\cat A})$). 

\subsection{}\label{GEN.10}
Let $A=\bigoplus_{i\in \ZZ}A_i$ be a $\ZZ$-graded unital $\kk$-algebra. We will always assume that
$A_i=0$ for all~$i<0$. In particular, this implies that $A_0\cong A/\bigoplus_{i>0} A_i$ as an algebra.
Clearly, each $A_i$ is an $A_0$-bimodule.

Given an $A_0$-bimodule~$V$, denote by $T^i_{A_0}(V)$ the $i$-fold tensor product 
$V\tensor_{A_0}\cdots\tensor_{A_0} V$.  Set $(K_A)_i=0$, for $i<0$, and 
$(K_A)_{i}=A_i$, for $i=0,1$, and define for~$i>1$
\begin{displaymath}
(K_A)_i=\bigcap_{j=0}^{i-2} T^j_{A_0}(A_1)\tensor_{A_0}(\ker m_A|_{A_1\tensor_{A_0}A_1})\tensor_{A_0} T^{i-j-2}_{A_0}(A_1),
\end{displaymath}
where $m_A:A\tensor_{A_0} A\to A$ is the multiplication map (cf.~\cite{BGS}*{\S2.6}). Clearly, all 
$(K_A)_i$, for $i\ge0$, are $A_0$-bimodules. In particular, $A_i\tensor_{A_0}-$, $(K_A)_i\tensor_{A_0}-$,
$\Hom_{A_0}(A_i,-)$ and $\Hom_{A_0}((K_A)_i,-)$ are endofunctors of the category of left $A_0$-modules.

For $i,j\in\ZZ_{\ge 0}$, denote by $m_{A_i,A_j}$ (respectively, $m_{A,A_j}$, $m_{A_i,A}$) the restriction of 
$m_A$ to $A_i\tensor_{A_0}A_j$ (respectively, to $A\tensor_{A_0}A_j$, $A_i\tensor_{A_0}A$). Let 
$K_A=\bigoplus_{i\ge 0} (K_A)_i$. The following Lemma is rather standard (cf.~\cite{BGS}).

\lemanch{GEN.10}
\begin{lem}
The object $K_A\tensor_{A_0} A$ (respectively, $A\tensor_{A_0}K_A$) can be equipped with the structure of 
a bounded above complex of $\ZZ$-graded right (respectively, left) $A$-modules, called the {\em right} 
(respectively, {\em left}) {\em Koszul complex} of~$A$, with 
\begin{displaymath}
(K_A\tensor_{A_0} A)^{-i}{}_j:=(K_A)_i\tensor_{A_0} A_{j-i} \quad
\text{ and }\quad (A\tensor_{A_0} K_A)^{-i}{}_j:=A_{j-i}\tensor_{A_0} (K_A)_i,
\end{displaymath} 
for $i,j\in\ZZ_{\ge 0}$, and with the  differential~$d^A$ (respectively, ${}^A d$) defined  by 
\begin{displaymath}
d^{A}{}_{-i}:=\id_{A_1}^{\tensor(i-1)}\tensor m_{A_1,A}\quad
\text{ and }\quad {}^Ad_{-i}:=m_{A,A_1}\tensor\id_{A_1}^{\tensor(i-1)},\quad\text{ for } i>0.
\end{displaymath}
\end{lem}

\begin{proof}
We only prove the statement for $K_A\tensor_{A_0}A$, the proof for $A\tensor_{A_0}K_A$ being similar.
Since $(K_A)_i\subset (K_A)_{i-1}\tensor_{A_0}A_1$, the map $d^A{}_{-i}$ is well-defined, is evidently a 
homomorphism of $\ZZ$-graded right $A$-modules and maps 
$(K_A\tensor_{A_0}A)^{-i}{}_j$ to $(K_A\tensor_{A_0}A)^{-i+1}{}_j$.
It remains to observe that 
\begin{displaymath}
d^A{}_{-i+1}\circ d^A{}_{-i}=\id_{A_1}^{\tensor (i-2)}\tensor m_{A_1,A}(\id_{A_1}\tensor m_{A_1,A})=
\id_{A_1}^{\tensor (i-2)}\tensor m_{A_1,A}(m_{A_1,A_1}\tensor\id_{A}) 
\end{displaymath}
which equals zero on~$(K_A)_i\tensor_{A_0}A$ since 
$(K_A)_i\subset T_{A_0}^{i-2}(A_1)\tensor_{A_0}\ker m_{A_1,A_1}$.
\end{proof}

\subsection{}\label{subsecnew101}
Fix a strictly full, additive, idempotent split subcategory~$\cat C$ of the category of left 
$A_0$-modules. We will always assume that $\cat C$ is contained in a strictly full abelian 
subcategory $\widehat{\cat C}$ of the category of left $A_0$-modules such that the 
following conditions hold for all $i>0$ and for all $X,X'\in\cat C$:
\begin{enumerate}[(I)]
%\item\label{cnd:C1}
%For all $X$, $Y$ 
%
%\hskip-\itemindent\hskip-\labelwidth
%and the following 
%conditions hold for all $i>0$ and for all $X,X'\in\cat C$:
\item\label{cnd:C1} $\dim_\kk\Hom_{A_0}(X,X')<\infty$;
\item\label{cnd:C2} $\Ext^i_{\widehat{\cat C}}(X,X')=0$;
\item\label{cnd:C3}
$A_i\tensor_{A_0}X$ and $(K_A)_i\tensor_{A_0} X$ are objects in~$\cat C$;
\item\label{cnd:C4}
$\Hom_{A_0}(A_i,X)$ and $\Hom_{A_0}((K_A)_i, X)$ are objects in~$\cat C$.
\end{enumerate}
In particular, the category~$\cat C$ is exact with all short exact sequences being split. This implies that 
$\cat C$ is abelian if and only if it is semisimple.

Denote by $\Gr_A\cat C$ the subcategory of $\Gr\cat C$ whose objects are graded $A$-modules and morphisms 
are those morphisms in~$\Gr\cat C$ which are also homomorphisms of~$A$-modules. 
Thus, for~$\mathsf M\in\Gr_A\cat C$, we have
\begin{displaymath}
\mathsf M=\bigoplus_{i\in\ZZ} \mathsf M_i,\qquad\mathsf M_i\in\cat C,\qquad A_i \mathsf M_j\subset \mathsf M_{i+j} 
\end{displaymath}
and for objects $\mathsf M$, $\mathsf N$ in~$\Gr_A\cat C$ the space of morphisms from~$\mathsf M$ to~$\mathsf N$ is
\begin{displaymath}
\hom_{A}(\mathsf M,\mathsf N):=\Hom_A(\mathsf M,\mathsf N)\cap \Hom_{\Gr\cat C}(\mathsf M,\mathsf N).
\end{displaymath}
We note that $\Gr_A\cat C$ is neither full nor isomorphism closed in $\Gr\cat C$, in general.
Let $\Gr_A^-\cat C$ (respectively, $\Gr_A^+\cat C$) be the full subcategory of bounded above (respectively,
bounded below) objects in~$\Gr_A\cat C$.  Both $\Gr_A^-\cat C$ and $\Gr_A^+\cat C$ are exact catgories
and hence we may consider projective objects in these categories relative to the corresponding exact 
structures. 

\subsection{}\label{GEN.30}
Given an object $X$ in~$\cat C$, let $\mathsf P^0(X):=A\tensor_{A_0}X$. By~\eqref{cnd:C3}, 
$\mathsf P^0(X)$ is an object in~$\Gr_A^+\cat C$ with $\mathsf P^0(X)_i=A_i\tensor_{A_0}X$, and there is 
a canonical epimorphism 
\begin{displaymath}
\mathsf P^0(X)\twoheadrightarrow\gr_0(X), \quad a\tensor x\mapsto a\lact x. 
\end{displaymath}
Dually, set $\mathsf I^0(X):=\displaystyle\bigoplus_{i\in\ZZ}\Hom_{A_0}(A_i,X)$ which is naturally a left $A$-submodule of 
$\Hom_{A_0}(A,X)$. By \eqref{cnd:C4}, $\mathsf I^0(X)$ is an object  in $\Gr_A^-\cat C$ with 
$\mathsf I^0(X)_{-i}=\Hom_{A_0}(A_i,X)$, and we have a natural monomorphism $\gr_0(X)\to \mathsf I^0(X)$ 
given by the isomorphism $X\cong\Hom_{A_0}(A_0,X)$ of left $A_0$-modules.

\lemanch{GEN.30}
\begin{lem}
\begin{enumerate}[{\rm(a)}]
\item\label{lem:GEN.30.a}
The assignment $X\mapsto \mathsf P^0(X)$, where $X\in\cat C$, defines an additive functor $\mathsf P^0:\cat C\to \Proj(\Gr_A^+\cat C)$.
\item\label{lem:GEN.30.b}
Similarly, the assignment $X\mapsto \mathsf I^0(X)$, where $X\in\cat C$, defines an additive functor $\mathsf I^0:\cat C\to \Inj(\Gr_A^-\cat C)$. 
\end{enumerate}
\label{lem:GEN.30}
\end{lem}

\begin{proof}
Clearly, $X\mapsto \mathsf P^0(X)$ defines a functor $\mathsf P^0:\cat C\to \Gr_A \cat C$ (in fact, 
$\cat C\to \Gr_A^+\cat C$). Let~$X\in\cat C$ and $\mathsf Y\in\Gr_A\cat C$. Since 
\begin{displaymath}
\Hom_{A}( A\tensor_{A_0} X,\mathsf Y)\cong \Hom_{A_0}(X,\Hom_A(A,\mathsf Y))\cong\Hom_{A_0}(X,\mathsf Y), 
\end{displaymath}
we conclude that $\hom_A(\mathsf P^0(X),\mathsf Y)\cong \Hom_{A_0}(X,\Pi_0(\mathsf Y))$, that is, the functor 
$\mathsf P^0$ is left adjoint to~$\Pi_0$. Since a left adjoint of an exact functor maps projectives to projectives
and every object in~$\cat C$ is projective, part~\eqref{lem:GEN.30.a} follows. 

Similarly, $X\mapsto \mathsf I^0(X)$ defines a functor $\mathsf I^0:\cat C\to \Gr_A^-\cat C$. 
Since, for any $\mathsf Y\in\Gr_A\cat C$, we have
\begin{displaymath}
\Hom_A(\mathsf Y,\Hom_{A_0}(A,X))\cong \Hom_{A_0}(A\tensor_A \mathsf Y,X)\cong \Hom_{A_0}(\mathsf Y,X),
\end{displaymath}
we conclude that $\hom_A(\mathsf Y,\mathsf I^0(X))\cong \Hom_{A_0}(\Pi_0(\mathsf Y),X)$. Thus, $\mathsf I^0$ is  
right adjoint to the exact functor~$\Pi$ and hence maps injectives to injectives. Since every object of~$\cat C$ is injective, this proves part~\eqref{lem:GEN.30.b}.
\end{proof}

\subsection{}\label{GEN.70}
Our present goal is to construct  functors 
\begin{displaymath}
\Gr_A^-\cat C\to \mathcal K^+(\Inj(\Gr_A^-\cat C))\quad\text{ and }\quad
\Gr_A^+\cat C\to\mathcal K^-(\Proj(\Gr_A^+\cat C)).
\end{displaymath}

Given an object~$\mathsf X=\bigoplus_{j\in\ZZ}\mathsf X_j$ in~$\Gr_A^-\cat C$, set, 
for all $r\ge 0$,
\begin{equation}\label{eq:I_r defn}
\begin{split}
\mathsf I^r(\mathsf X)&=\bigoplus_{i\in\ZZ}\Hom_{A_0}((K_A)_r\tensor_{A_0} A_i,\mathsf X)\cong
\bigoplus_{i\in\ZZ}\Hom_{A_0}(A_i,\Hom_{A_0}((K_A)_r,\mathsf X))\\
&=\bigoplus_{j\in\ZZ} \mathsf I^0(\Hom_{A_0}((K_A)_r,\Pi_j(\mathsf X))\ds{r-j}, 
\end{split}
\end{equation}
which is an object in $\Inj(\Gr_A^-\cat C)$ by \eqref{cnd:C4} 
and \lemref{GEN.30}\eqref{lem:GEN.30.b}. In particular,
\begin{equation}\label{eq:I_r component}
\mathsf I^r(\mathsf X)_{s}=\bigoplus_{i,j\in\ZZ\,:\, j-i=r+s}
\Hom_{A_0}((K_A)_r\tensor_{A_0} A_i,\mathsf X_j),\qquad s\in\ZZ.
\end{equation}
Define $d_{\mathsf X,r},\rho_{\mathsf X,r}:\mathsf I^r(\mathsf X)\to \mathsf I^{r+1}(\mathsf X)$ by 
\begin{displaymath}
d_{\mathsf X,r}(f)=f\circ d^A{}_{-r-1},\qquad 
\rho_{\mathsf X,r}(f)=\lact_{\mathsf X} \circ 
(\id_{A_1}\tensor f),\quad f\in \mathsf I^r(\mathsf X),
\end{displaymath}
where $\lact_{\mathsf X}:A\tensor_{A_0}\mathsf X\to\mathsf X$ is the action map 
$a\tensor x\mapsto a\lact x$. Note that $\rho_{\mathsf X,r}$ is well-defined since 
$(K_A)_{r+1}\subset A_1\tensor_{A_0} (K_A)_r$.

\lemanch{GEN.70}
\begin{lem}
For all $r\in\ZZ_{\ge 0}$ and $\mathsf X\in\Gr_A\cat C$, the map
$\rho_{\mathsf X,r}$ is a morphism in~$\Gr_A\cat C$ satisfying
$\rho_{\mathsf X,r+1}\circ \rho_{\mathsf X,r}=0$ and $\rho_{\mathsf X,r+1}\circ d_{\mathsf X,r}=d_{\mathsf X,r+1}\circ\rho_{\mathsf X,r}$.
In particular, $(\mathsf I^\bullet(\mathsf X),\partial^+_{\mathsf X})$, where $\partial^+_{\mathsf X,r}:=d_{\mathsf X,r}+(-1)^r \rho_{\mathsf X,r}$, is
an object in $\mathcal K^+(\Inj(\Gr_A^- \cat C))$.
\end{lem}

\begin{proof}
Clearly, $\rho_{\mathsf X,r}$ is a homomorphism of $A$-modules and 
\begin{displaymath}
\rho_{\mathsf X,r}(\Hom_{A_0}((K_A)_r\tensor_{A_0}A_i,\mathsf X_j))\subset 
\Hom_{A_0}((K_A)_{r+1}\tensor_{A_0}A_i,\mathsf X_{j+1}), 
\end{displaymath}
hence $\rho_{\mathsf X,r}\in\hom_A(\mathsf I^r(\mathsf X),
\mathsf I^{r+1}(\mathsf X))$ by~\eqref{eq:I_r component}.
Furthermore, for 
\begin{displaymath}
f\in\Hom_{A_0}((K_A)_r\tensor_{A_0}A_i,\mathsf X),
\end{displaymath}
we have 
\begin{displaymath}
\rho_{\mathsf X,r+1}\circ \rho_{\mathsf X,r}(f)=\lact_{\mathsf X}\circ(\id_{A_1}\tensor \lact_{\mathsf X})\circ(\id_{A_1}^{\tensor 2}\tensor f)=
\lact_{\mathsf X}\circ (m_{A_1,A_1}\tensor\id_{\mathsf X})\circ(  \id_{A_1}^{\tensor 2}\tensor f)=0
\end{displaymath}
since $(K_A)_{r+1}\subset \ker m_{A_1,A_1}\tensor_{A_0} T_{A_0}^{r-1}(A_1)$. Finally,
\begin{multline*}
\rho_{\mathsf X,r+1}\circ d_{\mathsf X,r}(f)=\rho_{\mathsf X,r+1}
(f\circ \id_{A_1}^{\tensor(r-1)}\tensor m_{A_1,A})=\lact_{\mathsf X}\circ (\id_{A_1}\tensor 
f\circ \id_{A_1}^{\tensor(r-1)}\tensor m_{A_1,A})\\
=\lact_{\mathsf X}\circ (\id_{A_1}\tensor f)\circ (\id_{A_1}^{\tensor r}\tensor m_{A_1,A})=
d_{\mathsf X,r+1}\circ \rho_{\mathsf X,r}(f).
\end{multline*}
The remaining assertion is obvious.
\end{proof}

\propanch{GEN.10}
\begin{prop}
The assignment $\mathsf X\mapsto (\mathsf I^\bullet(\mathsf X),\partial^+_{\mathsf X})$, 
$\mathsf X\in\Gr_A^-\cat C$ defines a functor 
\begin{displaymath}
\Phi^+:\Gr_A^-\cat C \to \mathcal K^+(\Inj(\Gr_A^-\cat C))
\end{displaymath}
which extends to an endofunctor $\Phi^+$ of $\mathcal K^+(\Inj(\Gr_A^-\cat C))$.
\end{prop}

\begin{proof}
For the first assertion, we only need to check that a morphism $\phi:\mathsf X\to \mathsf Y$ in~$\Gr_A^-\cat C$
gives rise to a morphism of complexes $\boldsymbol\phi:\mathsf I^\bullet(\mathsf X)\to\mathsf I^\bullet(\mathsf Y)$. 
Indeed, $\phi$ clearly induces morphisms $\phi_r:\mathsf I^r(\mathsf X)\to\mathsf I^r(\mathsf Y)$,
$f\mapsto \phi\circ f$, for $f\in\Hom_{A_0}((K_A)_r\tensor A_i,\mathsf X_j)$, and it is easy 
to see that $\partial^+_{\mathsf Y}\boldsymbol\phi=\boldsymbol\phi\partial^+_{\mathsf X}$.
Furthermore, given a bounded below complex $(\mathsf X^\bullet,\delta)$ over $\Gr_A^-\cat C$, 
define $\Phi^+((\mathsf X^\bullet,\delta))$ to be the total complex 
\begin{displaymath}
\Big(\bigoplus_{r+s=p} \mathsf I^r(\mathsf X^s),\sum_{r+s=p}\partial^+_{\mathsf X,r}+(-1)^s\delta_s\Big).
\end{displaymath}
Since $\mathsf X^s=0$ if~$s\ll 0$ and~$\mathsf I^r(\mathsf X^s)=0$ if~$r<0$,
we conclude that the above complex is bounded from below and, in fact, for each~$p$ the 
direct sum is finite. In particular, this implies that every term is an object in
$\Inj(\Gr_A^-\cat C)$. This yields the desired endofunctor on~$\mathcal K^+(\Inj(\Gr_A^-\cat C))$.
\end{proof}

\subsection{} \label{GEN.90}
The functor $\Phi^-:\Gr_A^+\cat C\to \mathcal K^-(\Proj(\Gr_A^+\cat C))$ is constructed similarly. 
Indeed, given~$\mathsf X=\bigoplus_{i\in\ZZ}\mathsf X_i\in\Gr_A^+\cat C$, we set,
for all $r\ge 0$,
\begin{equation}\label{eq:P_r defn}
\mathsf P^{-r}(\mathsf X)=A\tensor_{A_0} (K_A)_r\tensor_{A_0} \mathsf X=\bigoplus_{j\in\ZZ} \mathsf P^0((K_A)_r\tensor_{A_0}\Pi_j(\mathsf X))\ds{-j-r}.
\end{equation}
This implies that $\mathsf P^{-r}(\mathsf X)\in\Proj(\Gr_A^+\cat C)$ and 
\begin{equation}\label{eq:P_r component}
\mathsf P^{-r}(\mathsf X)_s=\bigoplus_{i,j\in\ZZ\,:\, i+j=s-r} A_i\tensor_{A_0} (K_A)_r\tensor_{A_0} \mathsf X_j.
\end{equation}
Define $\partial^-_{\mathsf X,-r}:\mathsf P^{-r}(\mathsf X)\to \mathsf P^{-r+1}(\mathsf X)$ by
\begin{displaymath}
\partial^-_{\mathsf X,-r}=(-1)^r\, {}^A d_{-r}\tensor \id_{\mathsf X}
+\id_A\tensor \id_{A_1}^{\tensor (r-1)}\tensor \lact_{\mathsf X}.
\end{displaymath}
Note that $\partial^-_{\mathsf X}$ is well-defined since $(K_A)_r\subset A_1\tensor_{A_0} (K_A)_{r-1}\cap (K_A)_{r-1}\tensor_{A_0} A_1$.

\lemanch{GEN.90}
\begin{lem}
We have $(\mathsf P^\bullet(\mathsf X),\partial^-_\mathsf X)\in\mathcal K^-(\Proj(\Gr_A^+\cat C))$ 
for all $\mathsf X\in\Gr_A^+\cat C$.
\end{lem}

\begin{proof}
It is obvious that $\partial^-_{\mathsf X,-r}$ is a homomorphism of $\ZZ$-graded $A$-modules. 
Setting $\lambda_{\mathsf X,-r}:=\id_A\tensor \id_{A_1}^{\tensor (r-1)}\tensor \lact_{\mathsf X}$, we obtain
\begin{multline*}
\lambda_{\mathsf X,-r+1}\circ \lambda_{\mathsf X,-r}=
\id_A\tensor\id_{A_1}^{\tensor(r-2)}\tensor(\lact_{\mathsf X}\circ (\id_{A_1}\tensor \lact_{\mathsf X})
=\\=\id_A\tensor\id_{A_1}^{\tensor(r-2)}\tensor\lact_{\mathsf X}\circ (m_{A_1,A_1}\tensor \id_{\mathsf X})=0
\end{multline*}
since $(K_A)_r\subset T_{A_0}^{r-2}(A_1)\tensor\ker m_{A_1,A_1}$. Furthermore, as we clearly have
\begin{displaymath}
({}^A d_{-r+1}\tensor \id_{\mathsf X})\circ \lambda_{\mathsf X,-r}=
\lambda_{\mathsf X,-r+1}\circ ({}^A d_{-r}\tensor \id_{\mathsf X}),
\end{displaymath}
it follows that $\partial^-_{\mathsf X,-r+1}\circ\partial^-_{\mathsf X,-r}=0$.
\end{proof}

The following statement is proved similarly to~\propref{GEN.10}.

\propanch{GEN.90}
\begin{prop}
The assignment $\mathsf X\mapsto (\mathsf P^\bullet(\mathsf X),\partial^-_{\mathsf X})$, 
for $\mathsf X\in\Gr_A^+\cat C$, defines a functor $\Phi^-:\Gr_A^+\cat C
\to \mathcal K^-(\Proj(\Gr_A^+\cat C))$ which further extends to an endofunctor $\Phi^-$
of $\mathcal K^-(\Proj(\Gr_A^+\cat C))$.
\end{prop}

\subsection{}\label{GEN.85}
Given a complex $(\mathsf X^\bullet,d_{\mathsf X})$ over~$\Gr_A\cat C$, its $i$-th homology 
$\mathsf H_i(\mathsf X^\bullet)$ is defined as $\mathsf H_i(\mathsf X^\bullet):=\ker d_{\mathsf X,i}/\im d_{\mathsf X,i-1}$ which, a priori, 
is an object in~$\Gr_A\widehat{\cat C}$ (see \subsref{subsecnew101}). We will need some properties 
of $\mathsf I^\bullet$ and $\mathsf P^\bullet$ which are collected in the following Lemma.

\lemanch{GEN.85}
\begin{lem}
Suppose that $A$ is generated by $A_1$ over~$A_0$. Then we have:
\begin{enumerate}[{\rm(a)}]
\item\label{lem:GEN.85.0}
$\mathsf H_0(\mathsf I^\bullet(\mathsf X))\cong\mathsf X$ (respectively, 
$\mathsf H_0(\mathsf P^\bullet(\mathsf Y))\cong \mathsf Y$)
for any $\mathsf X\in\Gr_A^-\cat C$ (respectively, $\mathsf Y\in\Gr_A^+\cat C$).

Furthermore, if $\mathsf X$ is concentrated at degree zero then
\item\label{lem:GEN.85.a}  
$\mathsf I^r(\mathsf X)=\bigoplus_{\mathsf Z} \mathsf I^0(\mathsf Z)\ds{r}$, i.e. $\mathsf I^\bullet$ is diagonal.
\item\label{lem:GEN.85.c}  
$\mathsf P^{-r}(\mathsf X)=\bigoplus_{\mathsf Z} \mathsf P^0(\mathsf Z)\ds{-r}$, i.e. $\mathsf P^\bullet$ is diagonal.
\item\label{lem:GEN.85.b}
$\mathsf H_r(\mathsf I^\bullet(\mathsf X))_{-r}=0=\mathsf H_{-r}(\mathsf P^\bullet(\mathsf X))_r$, for $r>0$.
\end{enumerate}
\end{lem}

\begin{proof}
To prove part~\eqref{lem:GEN.85.0}, note that $\partial^+_{\mathsf X,-1}=0$ and 
$f\in\mathsf I^0(\mathsf X)=\bigoplus_{i\in\ZZ} \Hom_{A_0}(A_i,\mathsf X)$ is 
in~$\ker\partial^+_{\mathsf X,0}$ if and only if $f(a_1 a)=a_1\lact f(a)$ for all $a_1\in A_1$, $a\in A$, and
$i\ge 0$. Since  $A$ is generated by~$A_1$, we have that $\ker\partial^+_{\mathsf X,0}$ is a $\ZZ$-graded 
$A$-submodule of~$\mathsf I^0(\mathsf X)$ and $f\in\ker\partial^+_{\mathsf X,0}$ is uniquely determined by $f(1)$, 
hence the map $\ker \partial^+_{\mathsf X,0}\to\mathsf X$, $f\mapsto f(1)$, is an isomorphism of $\ZZ$-graded 
$A$-modules. Similarly, $\partial^-_{\mathsf X,0}=0$ and the image of $\partial^-_{\mathsf X,-1}$ is generated 
by elements of the form $aa_1\tensor x-a\tensor a_1\lact x$, where $a\in A$, $a_1\in A_1$, and $x\in\mathsf X$. 
Since $A$ is generated by~$A_1$, we have that $\mathsf H_0(P^\bullet(\mathsf X))$ is isomorphic to 
$A\tensor_A \mathsf X\cong \mathsf X$.

Parts~\eqref{lem:GEN.85.a} and~\eqref{lem:GEN.85.c} follow from~\eqref{eq:I_r defn} and~\eqref{eq:P_r defn}, respectively. Further, by~\eqref{eq:I_r component}, we have
\begin{gather*}
\mathsf I^r(\mathsf X)_{-r}=\Hom_{A_0}((K_A)_r,\mathsf X),\quad 
\mathsf I^{r-1}(\mathsf X)_{-r}=\Hom_{A_0}((K_A)_{r-1}\tensor_{A_0}A_1,\mathsf X),\\
\mathsf I^{r+1}(\mathsf X)_{-r}=0. 
\end{gather*}
Thus, for $r>0$ it suffices to prove that $\partial^+_{\mathsf X,r-1}(\mathsf I^{r-1}(\mathsf X)_{-r})=
\mathsf I^r(\mathsf X)_{-r}$. But this is immediate since $(K_A)_r\subset (K_A)_{r-1}\tensor A_1$, the map
$\partial^+_{\mathsf X,r-1}|_{\mathsf I^{r-1}(\mathsf X)_{-r}}$ identifies with the restriction and $\mathsf X$ is injective in~$\mathscr C$. Similarly, by~\eqref{eq:P_r component}
\begin{displaymath}
\mathsf P^{-r-1}(\mathsf X)_r=0,\quad \mathsf P^{-r}(\mathsf X)_{r}=(K_A)_r\tensor_{A_0}\mathsf X,\qquad 
\mathsf P^{-r+1}(\mathsf X)_{r}=A_1\tensor_{A_0}(K_A)_{r-1}\tensor_{A_0}\mathsf X.
\end{displaymath}
Thus, we need to prove that $\partial^-_{\mathsf X,-r}|_{\mathsf P^{-r}(\mathsf X)_r}$ is injective. 
This follows since $\partial^-_{\mathsf X,-r}|_{P_{-r}(\mathsf X)_r}$ is induced by 
the natural inclusion $(K_A)_r\to A_1\tensor_{A_0} (K_A)_{r-1}$ and 
$\mathsf X$ is projective in~$\mathscr C$. This completes the proof of part~\eqref{lem:GEN.85.b}.
\end{proof}

\subsection{}\label{GEN.120}
Given an object~$\mathsf X$ in~$\Gr^-_A\cat C$,
define a complex $(\soc \mathsf I)^\bullet(\mathsf X)$ of left $A_0$-modules, and hence $A$-modules
via the canonical homomorphism of algebras $A\twoheadrightarrow A_0$, by
\begin{displaymath}
(\soc \mathsf I)^r(\mathsf X)=\Hom_{A_0}((K_A)_r,\mathsf X).
\end{displaymath}
Observe that~$\partial^+_{\mathsf X}$ restricts to
$(\soc \mathsf I)^\bullet(\mathsf X)$ since,
given $f\in\Hom_{A_0}((K_A)_r,\mathsf X)$, we have 
\begin{equation*}
\partial^+_{\mathsf X,r}(f)=(-1)^{r+1}\lact_{\mathsf X}\circ(\id_{A_1}\tensor f)\in \Hom_{A_0}((K_A)_{r+1},\mathsf X).
\end{equation*}
We can regard~$(\soc \mathsf I)^\bullet(\mathsf X)$ as a subcomplex of~$\mathsf I^\bullet(\mathsf X)$.
Clearly, this construction is functorial in~$\mathsf X$. 

Dually, define $(\Top \mathsf P)^\bullet(\mathsf X)$ by
\begin{displaymath}
(\Top \mathsf P)^{-r}(\mathsf X)=(K_A)_r\tensor_{A_0}\mathsf X.
\end{displaymath}
Again, the differential $\partial^-_{\mathsf X}$ restricts to a morphism 
$(\Top \mathsf P)^{-r}\to(\Top \mathsf P)^{-r+1}$ which is given by 
$\partial^-_{\mathsf X,-r}=\id_{A_1}^{\tensor r-1}\tensor \lact_{\mathsf X}$.
The canonical epimorphism $A\twoheadrightarrow A_0$ yields a commutative diagram
\begin{displaymath}
\xymatrix{
\mathsf P^{-r}(\mathsf X)\ar@{->>}[r]\ar[d]_{\partial^-_{\mathsf X,-r}}& (\Top \mathsf P)^{-r}(\mathsf X)\ar[d]^{\partial^-_{\mathsf X,-r}}\\
\mathsf P^{-r+1}(\mathsf X)\ar@{->>}[r]& (\Top \mathsf P)^{-r+1}(\mathsf X).}
\end{displaymath}
That is, $(\Top \mathsf P)^\bullet(\mathsf X)$ identifies with a quotient of
$\mathsf P^\bullet(\mathsf X)$. 

\section{Module algebras and semidirect products}\label{snew3}

\subsection{}\label{KDS10}
Let $B$ be a bialgebra over~$\kk$ with the comultiplication~$\Delta_B:B\to B\tensor_\kk B$ and the counit~$\varepsilon_B:B\to \kk$. Henceforth we write $\Delta_B(b)=b_{(1)}\tensor b_{(2)}$ in Sweedler's notation.
Let~$H_B$ be a right $B$-module algebra, that is~$H_B$ is an associative $\kk$-algebra and a right $B$-module, 
while the multiplication map $m_H:H_B\tensor_\kk H_B\to H_B$ is a morphism of right~$B$-modules and~$B$ 
acts on~$1_{H_B}$ by the counit. Thus,  for all $b\in B$ and $h,h'\in H_B$ we have 
\begin{displaymath}
(hh')\ract b=(h\ract b_{(1)})(h'\ract b_{(2)}),\qquad 1_{H_B}\ract b=\varepsilon_B(b)1_{H_B}.
\end{displaymath}
Then we can form the semidirect product 
$B\ltimes H_B$ as follows. As a $\kk$-vector space, $B\ltimes H_B=B\tensor_\kk H_B$, with the multiplication given by
\begin{displaymath}
(b\tensor h)\cdot(b'\tensor h')=bb'_{(1)}\tensor (h\ract b'_{(2)})h',\qquad b,b'\in B,\, h,h'\in H_B.
\end{displaymath}

Let~$M$ be a right $B$-module. Then~$M\tensor_\kk H_B$ acquires a natural structure of a right 
$B\ltimes H_B$-module via
\begin{displaymath}
(m\tensor h)\ract (b\tensor h')=m\ract b_{(1)}\tensor (h\ract b_{(2)})h'.
\end{displaymath}
Suppose that $H_B=\bigoplus_{i\in\ZZ} (H_B)_i$ is a $\ZZ$-graded algebra and that each $(H_B)_i$ is a 
right $B$-submodule of~$H_B$. Then $B\ltimes H_B$ is naturally a $\ZZ$-graded algebra with 
$(B\ltimes H_B)_i=B\tensor_\kk (H_B)_i$ as a $\kk$-vector space. Clearly, $M\tensor_\kk H_B$ is a 
$\ZZ$-graded $B\ltimes H_B$-module with $(M\tensor_\kk H_B)_i=M\tensor (H_B)_i$, where $i\in\ZZ$.

Similarly, if~${}_B H$ is a left $B$-module algebra, we define ${}_B H\rtimes B$ as ${}_B H\tensor_\kk B$ 
with the multiplication defined by
\begin{displaymath}
(h\tensor b)\cdot(h'\tensor b')=h(b_{(1)}\lact h')\tensor b_{(2)}b',\qquad b,b'\in B,\,h,h'\in {}_B H.
\end{displaymath}
Let~$M$ be a left $B$-module. Then~${}_B H\tensor_\kk M$ acquires the natural structure of a left 
${}_B H\rtimes B$-module via
\begin{displaymath}
(h\tensor b)\lact(h'\tensor m)=h(b_{(1)}\lact h')\tensor b_{(2)}\lact m,\qquad h,h'\in {}_B H,\, b\in B,\, m\in M.
\end{displaymath}
Finally, if ${}_B H=\bigoplus_{i\in\ZZ} ({}_B H)_i$ is a $\ZZ$-graded algebra such that each $({}_B H)_i$ is a left 
$B$-submodule of~${}_B H$, then ${}_B H\rtimes B$ has the natural structure of a $\ZZ$-graded algebra given by 
$({}_B H\rtimes B)_i=({}_B H)_i\tensor_\kk B$. Likewise, ${}_B H\tensor_\kk M$ is a graded 
${}_B H\rtimes B$-module with $({}_B H\tensor_\kk M)_i={}_B H_i\tensor_\kk M$.

\subsection{}\label{KDS15}
Suppose now that our graded algebra~$A$ from~\subsref{GEN.10} is isomorphic to~$A_0\ltimes H$ where 
$A_0$ is a bialgebra with the comultiplication~$\Delta$ and the counit~$\varepsilon$ as above and 
$H$ is a $\ZZ$-graded {\em right} $A_0$-module algebra with $H_0=\kk$ and $H_i=0$, $i<0$.
Thus, in particular, $A_i\cong A_0\tensor_\kk H_i$ as an $A_0$-bimodule, where 
\begin{displaymath}
a\lact(a'\tensor h)\ract a''=aa'a''_{(1)}\tensor h\ract a''_{(2)}
\end{displaymath}
for all $a,a',a''\in A_0$, $h\in H$. 
Clearly, we have
\begin{displaymath}
(K_A)_i\cong A_0\tensor_\kk (K_H)_i,\qquad (K_H)_i=\bigcap_{j=0}^{i-2} T^j_\kk(H_1)\tensor_\kk \ker m_{H_1,H_1}\tensor_{\kk} T^{i-j-2}_\kk(H_1)
\end{displaymath}
as $A_0$-bimodules. The action of~$A_0$ can be written explicitly as
\begin{displaymath}
a''\lact (a'\tensor h_{1}\tensor\cdots\tensor h_{r})\ract a=a'' a' a_{(1)}\tensor h_{1}\ract a_{(2)}\tensor \cdots
\tensor h_{r}\ract a_{(r+1)},
\end{displaymath}
where $a,a',a''\in A_0$, $(\Delta\tensor 1^{\tensor r-1})\circ\cdots\circ
(\Delta\tensor 1)\circ\Delta(a)=a_{(1)}\tensor\cdots\tensor a_{(r+1)}$  and, finally, 
$h_1\tensor \cdots\tensor h_r\in (K_H)_r$, with $h_i\in H_1$, in Sweedler-like notation with summation understood. 

\subsection{}\label{KDS20}
Our basic example of this construction is provided by the following generalization of the 
Takiff Lie algebra (cf.~\cite{Tak}) and is motivated by results of~\cite{CG-2}.
Let $\lie g$ be a Lie algebra and~$V$ be a $\lie g$-module. Then we can consider the 
Lie algebra $\lie g\ltimes V$, which is isomorphic to~$\lie g\oplus V$ as a vector space, 
with the Lie bracket defined by
\begin{displaymath}
[(x,v),(x',v')]=([x,x']_{\lie g},xv'-x'v),\qquad x,x'\in\lie g,\, v,v'\in V.
\end{displaymath}
Then $H=S(V)$ is a (right) $U(\lie g)$-module algebra. Let~$A_0=U(\lie g)$. Then it is 
easy to see, using the Poincar\'e-Birkhoff-Witt theorem, that 
$U(\lie g\ltimes V)\cong U(\lie g)\ltimes S(V)$. The classical Takiff Lie algebra is obtained by 
taking $V$ to be the adjoint representation of~$\lie g$.

Similarly, we can consider a generalized {\em Takiff Lie superalgebra} 
$\lie s=\lie s_0\oplus \lie s_1$ with $\lie s_0=\lie g$,
$\lie s_1=V$ and the (super) Lie bracket defined by
\begin{displaymath}
[(x,v),(x',v')]=([x,x']_{\lie g},xv'+x'v),\qquad x,x'\in\lie s_0,\,v,v'\in\lie s_1.
\end{displaymath}
In this case, regarding $\twedge V$ as a left $U(\lie g)$-module, we have
\begin{displaymath}
U(\lie s)\cong \twedge V\rtimes U(\lie g).
\end{displaymath}

\subsection{}\label{KDS25}
Let~$R$ be a unital ring and $K$ a unital subring of its center. Let $V$
be a left $R$-module. Then $V^*=\Hom_K(V,K)$ is naturally a right $R$-module via 
\begin{displaymath}
(f\ract r)(v)=f(r\lact v),\quad \text{ for }\quad f\in V^*, v\in V, r\in R. 
\end{displaymath}
Similarly, if~$W$ is a right $R$-module, $W^*=\Hom_K(W,K)$ is a left $R$-module via 
$(r\lact f)(w)=f(w\ract r)$, where $f\in W^*$, $w\in W$, $r\in R$. The following 
lemma is standard.

\begin{lem}\lemanch{KDS25}
Assume that~$R\cong R^*$ as an $R$-bimodule. Suppose that~$W$ is a 
finitely generated projective left $R$-module and $V$ is an~$R$-bimodule. Then
\begin{displaymath}
(V\tensor_R W)^*\cong W^*\tensor_R V^*
\end{displaymath}
\end{lem}

\begin{proof}
We have
\begin{displaymath}
(V\tensor_R W)^*=\Hom_K(V\tensor_R W,K)\cong\Hom_R(W,\Hom_K(V,K))\cong \Hom_R(W,R)\tensor_R V^*
\end{displaymath}
since~$W$ is projective and finitely generated. Since~$R\cong \Hom_K(R,K)$ as an $R$-bimodule,
\begin{equation*}
\Hom_R(W,R)\cong \Hom_R(W,\Hom_K(R,K))\cong \Hom_K(R\tensor_R W,K)=W^*.\qedhere
\end{equation*}
\end{proof}

\subsection{}\label{KDS30}
Assume that~$H_1$ is finite dimensional over~$H_0=\kk$. In particular, this implies that 
$(H_1\tensor_\kk H_1)^*\cong H_1^*\tensor_\kk H_1^*$ where $H_1^*=\Hom_\kk(H_1,\kk)$
(cf. \lemref{KDS25}). Define the quadratic dual~$H^!$ of~$H$ by
\begin{displaymath}
H^!=T_\kk(H_1^*)/\langle \im \mu^*\rangle
\end{displaymath}
where $\mu^*:H_2^*\to H_1^*\tensor_\kk H_1^*$ is
induced by the multiplication map, namely
\begin{displaymath}
\mu^*(\xi)(h\tensor h')=\xi(hh'),\qquad h,h'\in H_1,\,\xi\in H_2^*.
\end{displaymath}
By construction, $H^!$ is $\ZZ$-graded with $H^!{}_i$ being the canonical image of 
$T_\kk^i(H_1^*)$. In particular, $H^!{}_i=0$ if~$i<0$.

\begin{rem}
After \lemref{KDS25}, we can consider a more general situation. Namely, we can 
assume that $H_0\cong H_0^*$ as an $H_0$-bimodule and that $H_1$ is finitely generated 
and projective as an~$H_0$-module. Then $A_0=B\ltimes H_0$ and  we still 
have 
\begin{displaymath}
(H_1\tensor_{H_0}H_1)^*\cong H_1^*\tensor_{H_0} H_1^*.
\end{displaymath}
\end{rem}

Given a bialgebra $B$, denote by $B^{\cop}$ the vector space~$B$ endowed with the same 
multiplication and with the opposite comultiplication.

\lemanch{KDS30}
\begin{lem}
$H^!$ is a graded left $A_0^{\cop}$-module algebra,
that is, $a\lact (\xi\xi')=(a_{(2)}\lact \xi)(a_{(1)}\lact \xi')$.
\end{lem}

\begin{proof}
Using~\subsref{KDS25}, we obtain
\begin{align*}
(a\lact (\xi\tensor \xi'))(h'\tensor h)&=(\xi\tensor \xi')((h'\tensor h)\ract a)=(\xi\tensor \xi')(h'\ract a_{(1)}\tensor h\ract a_{(2)})\\
&=\xi(h\ract a_{(2)})\xi'(h'\ract a_{(1)})=(a_{(2)}\lact\xi)(h)(a_{(1)}\lact \xi')(h')\\
&=(a_{(2)}\lact\xi\tensor a_{(1)}\lact \xi')(h'\tensor h),
\end{align*} 
for all $\xi\in T_\kk^r(H_1^*)$, $\xi'\in T_\kk^s(H_1^*)$, $h\in T_\kk^r(H_1)$ 
and~$h'\in T_\kk^s(H_1)$. This shows that $T_\kk(H_1^*)$ is a left $A_0^{\cop}$-module 
algebra and the action is obviously compatible with the grading. Thus, it remains to check 
that $\im m_H^*\subset H_1^*\tensor_\kk H_1^*$ is a left $A_0^{\cop}$-submodule of 
$H_1^*\tensor_\kk H_1^*$. But this is immediate since~$m_H$ is a morphism of right $A_0$-modules. 
\end{proof}

\subsection{}\label{KDS50}
Assume that $H$ is quadratic, that is, $H$ is a quotient of $T_\kk(H_1)$ by the ideal generated by 
$\ker m_{H_1,H_1}$. Note that $H^!$ is always quadratic. Since $H_1$ is finite dimensional, it follows 
that $H_i$ is finite dimensional for all~$i\ge 0$. Then $(H^!)^!$ naturally identifies with~$H$. Moreover, 
by~\cite{BGS}*{\S2.8} we have $(K_{H})^*_i\cong H^!{}_i$ and $(K_{H^!})_i\cong H_i^*$. By~\lemref{KDS25}, 
there is a natural isomorphism 
\begin{displaymath}
\psi_{i,j}:(T_\kk^j(H_1)\tensor_\kk T_\kk^i(H_1))^*\to 
(T_\kk^i(H_1))^*\tensor_\kk (T_\kk^j(H_1))^*\cong T_\kk^i(H_1^*)\tensor_\kk T_\kk^j(H_1^*) 
\end{displaymath}
which yields an isomorphism 
\begin{displaymath}
\bar\psi_{i,j}:((K_H)_j\tensor_\kk H_i)^*\to (H_i)^*\tensor_\kk ((K_H)_j)^*\to (K_{H^!})_i\tensor_\kk H^!{}_j.
\end{displaymath}

Let $m_{j,i}:T_\kk^j(H_1)\tensor_\kk H_i\to T_\kk^{j-1}(H_1)\tensor_\kk H_{i+1}$ be the map induced by 
the natural isomorphism $T_\kk^j(H_1)\tensor T_\kk^i(H_1)\to T_\kk^{j-1}(H_1)\tensor T_\kk^{i+1}(H_1)$ 
and denote by $\bar m_{j,i}$ its restriction to $(K_H)_j\tensor_\kk H_i$. Since $(K_H)_j\subset 
(K_H)_{j-1}\tensor_\kk H_1$, it follows that $\bar m_{j,i}((K_H)_j\tensor_\kk H_i)\subset (K_H)_{j-1}
\tensor_\kk H_{i+1}$. The corresponding maps for $H^!$ are denoted by $\bar m_{j,i}^!$.

\lemanch{KDS50}
\begin{lem}
We have $\bar\psi_{i-1,j+1}(f\circ\bar m_{j+1,i-1})=\bar m^!_{i,j}\circ\bar\psi_{i,j}(f)$ 
for all $f\in ((K_H)_j\tensor_\kk H_i)^*$, $i> 0$, and $j\ge0$.
\end{lem}

\begin{proof}
Take $ u\tensor h\in (K_H)_{j+1}\tensor H_{i-1}$. Since $(K_H)_{j+1}\subset (K_H)_j\tensor_\kk H_1$, 
we can write $u=u_1\tensor u_2$ in Sweedler-like notation, where $u_1\in (K_H)_j$, $u_2\in H_1$.
Then we have $\bar m_{j+1,i-1}(u\tensor h)=u_1\tensor u_2 h$ and, for any $f\in ((K_H)_j\tensor_\kk H_i)^*$, 
we have 
\begin{displaymath}
\psi_{i-1,j+1}(f\circ \bar m_{j+1,i-1})(u\tensor h)=f_1(u_1)f_2(u_2 h), 
\end{displaymath}
where 
$\bar\psi_{i,j}(f)=f_2\tensor f_1$ in Sweedler-like notation, with $f_1\in (K_H)_j{}^*$ and $f_2\in 
H_i{}^*$. Furthermore, identifying $H_i{}^*$ with $(K_{H^!})_i\subset (K_{H^!})_{i-1}\tensor 
H^!{}_1\cong H_{i-1}{}^*\tensor H_1{}^*$, we can write $\bar\psi_{i,j}(f)=f_3\tensor f_2\tensor f_1$, 
where $f_3\in H_{i-1}{}^*$, $f_2\in H_1{}^*$ and $f_1\in (K_H)_j{}^*$, whence 
$\psi_{i-1,j+1}(f\circ \bar m_{j+1,i-1})(u\tensor h)=f_1(u_1)f_2(u_2)f_3(h)$.

On the other hand, $\bar m^!_{i,j}$ maps $(K_{H^!})_i\tensor_\kk H^!{}_j\cong H_i{}^*\tensor_\kk (K_H)_j{}^*$ 
to 
\begin{displaymath}
(K_{H^!})_{i-1}\tensor H^!_{j+1}\cong H_{i-1}{}^*\tensor (K_H)_{j+1}{}^*.
\end{displaymath}
Thus, 
\begin{displaymath}
\bar m^!_{i,j}\circ \bar\psi_{i,j}(f)(u\tensor h)=(f_3\tensor f_2\tensor f_1)(u_1\tensor u_2\tensor h)=
f_1(u_1)f_2(u_2)f_3(h). \qedhere
\end{displaymath}
\end{proof}

\subsection{}\label{KDS70}
The algebra~$T_\kk(H_1^*)$ acts naturally on $T_\kk(H_1)$ by right contractions. Namely,
$\xi\in H_1^*$ acts trivially on~$T_\kk^0(H_1)$ and by $\id_{H_1}^{\tensor r-1}\tensor \xi$, for
$\xi\in H_1^*$, on~$T^r_\kk(H_1)$. Since the algebra~$T_\kk(H_1^*)$ is freely generated by $H_1^*$,
this extends to a right action of~$T_\kk(H_1^*)$ which we will denote by $(u,\xi)\mapsto u\triangleleft \xi$,
where $u\in T_\kk(H_1)$, $\xi\in T_\kk(H_1^*)$.

Similarly, the algebra $T_\kk(H_1)$ acts on~$T_\kk(H_1^*)$ by left contractions. Namely,
$H_1$ acts trivially on $T_\kk^0(H_1^*)$ and by 
$\langle h,\cdot\rangle\tensor \id_{H_1^*}{}^{\tensor (r-1)}$, for $h\in H_1$, on $T_\kk^r(H_1^*)$, where 
we have $\langle h,\xi\rangle=\xi(h)$, for $h\in H_1$, $\xi\in H_1^*$. Since $T_\kk(H_1)$ is freely generated by 
$H_1$, this extends to a left $T_\kk(H_1)$-action on~$T_\kk(H_1^*)$.

\lemanch{KDS70}
\begin{lem}
\begin{enumerate}[{\rm(a)}]
\item\label{lem:MA70.a} The right action of~$T_\kk(H_1^*)$ on~$T_\kk(H_1)$ by right contractions induces a 
right action of the algebra~$H^!$ on $K_H:=\bigoplus_{r\in\ZZ}(K_H)_r$.
\item\label{lem:MA70.b} If~$H$ is quadratic, then the left action of $T_\kk(H_1)$ on~$T_\kk(H_1^*)$ by 
left contractions induces a left action of $H$ on~$K_{H^!}$.
\end{enumerate}
\end{lem}

\begin{proof}
To prove part~\eqref{lem:MA70.a}, observe that we have $(1^{\tensor r-1}\tensor \xi)(u)\in (K_H)_{r-1}$
for all elements $u\in (K_H)_r$. Thus, the algebra $T_\kk(H_1^*)$ acts on~$K_H$. It remains to prove that
the defining ideal of~$H^!$ acts trivially. For this, choose any $\xi\in H_2^*$ and observe that 
$m_H^*(\xi)(h\tensor h')=\xi(hh')$, hence~$m_H^*(\xi)(\ker m_H)=0$. Part~\eqref{lem:MA70.b} is proved similarly.
\end{proof}

\section{Koszul duality for semidirect products}\label{snew4}

\subsection{}\label{KDS90}
Retain the assumptions of~\subsref{KDS30}. Using \lemref{KDS30}, define $A^\circledast:=H^!\rtimes A_0^{\cop}$.
Our aim now is to construct an action of $A^\circledast$ on the complex $(\soc \mathsf I)^\bullet(\mathsf X)$, 
$\mathsf X\in\Gr_A^-\cat C$
where $\cat C$ is a category of left $A_0$-modules satisfying \eqref{cnd:C1}--\eqref{cnd:C4}. First, we would like to identify the left $A=A_0\ltimes H$-module  
\begin{displaymath}
\mathsf I^r(\mathsf X)=\bigoplus_{i\in\ZZ}\Hom_{A_0}((K_A)_r\tensor_{A_0} A_i,\mathsf X)
\end{displaymath}
with a more manageable object.
We have
\begin{displaymath}
(K_A)_r\tensor_{A_0}A_i\cong A_0\tensor_\kk (K_H)_r\tensor_{A_0} A_0\tensor_\kk H_i\cong 
A_0\tensor_\kk (K_H)_r\tensor_\kk H_i
\end{displaymath}
as $A_0$-bimodules, where $A_0$ acts on $A_0\tensor_\kk (K_H)_r\tensor_\kk H_i$ via
\begin{equation}\label{eq:MA90.10}
a\lact (a''\tensor u\tensor h)\ract a'=a a''a'_{(1)}\tensor u\ract a'_{(2)}\tensor h\ract a'_{(3)},\quad 
a,a',a''\in A_0,\, u\in (K_H)_r,\, h\in H_i.
\end{equation}
Then
\begin{displaymath}
\Hom_{A_0}((K_A)_r\tensor_{A_0} A_i,\mathsf X)\cong\Hom_\kk((K_H)_r\tensor_\kk H_i,\mathsf X).
\end{displaymath}
In particular, $(\soc \mathsf I)^\bullet(\mathsf X)$ identifies with 
\begin{displaymath}
\bigoplus_{r\in\ZZ} \Hom_{\kk}((K_H)_r,\mathsf X).
\end{displaymath}

The right action of~$H^!$ on~$K_H$ by contractions introduced in~\subsref{KDS70} yields a left action 
of~$H^!$ on $\Hom_\kk(K_H,\mathsf X)$ which we denote by $\blact$; thus, 
$(\xi\blact f)(u_1\tensor u_2)=\xi(u_2)f(u_1)$, where $u_1\tensor u_2\in K_H$ in Sweedler-like notation, 
with $u_2\in H_1$. Clearly, $(\soc \mathsf I)^\bullet(\mathsf X)$ identifies with a left $H^!$-submodule 
of~$\Hom_\kk(K_H,\mathsf X)$. Consider the natural left action of $A_0^{\cop}$ on $\Hom_\kk(K_H,\mathsf X)$ given by
\begin{displaymath}
(a\blact f)(u)=a_{(1)} f(u\ract a_{(2)}),\qquad f\in\Hom_\kk(K_H,\mathsf X),\, u\in K_H,\, a\in A_0,
\end{displaymath}
where the right action of~$A_0$ on~$K_H$ is defined by~\eqref{eq:MA90.10}.

\begin{lem}
The left actions of~$A_0^{\cop}$ and~$H^!$ on~$\Hom_\kk(K_H,\mathsf X)$ defined above yield a left 
action of~$A^\circledast$ on $\Hom_\kk(K_H,\mathsf X)$. Moreover, $(\soc \mathsf I)^\bullet(\mathsf X)$ 
is a left $A^\circledast$-submodule of~$\Hom_\kk(K_H,\mathsf X)$ and the $A^\circledast$ action  commutes
with the differential~$\partial^+_{\mathsf X}$.
\end{lem}

\begin{proof}
To prove the first statement, it suffices to show that 
\begin{displaymath}
a\blact (\xi\blact f)=(a_{(2)}\lact\xi)\blact (a_{(1)} \blact f),
\end{displaymath}
for all~$f\in\Hom_{\kk}(K_H,\mathsf X)$,  $a\in A_0$ and $\xi\in H_1^*$. 

Take~$u\in K_H$ and write $u=u_1\tensor u_2\in K_H$ in Sweedler-like notation, with~$u_2\in H_1$. Then we have 
\begin{align*}
(a\blact (\xi\blact f))(u)&=a_{(1)}(\xi\blact f)(u_1\ract a_{(2)}\tensor u_2\ract a_{(3)})
=\xi(u_2\ract a_{(3)}) a_{(1)} f(u_1\ract a_{(2)})\\&=(a_{(3)}\lact\xi)(u_2) a_{(1)} f(u_1\ract a_{(2)})
=(a_{(2)}\lact \xi)(u_2) (a_{(1)}\blact f)(u_1)\\
&=((a_{(2)}\lact \xi)\blact (a_{(1)}\blact f))(u).
\end{align*}
Since all actions involved are locally finite, it follows that~$(\soc \mathsf I)^\bullet(\mathsf X)$ is a left $A^{\circledast}$-submodule of~$\Hom_\kk(K_H,\mathsf X)$.
For the second assertion, we need to show that 
\begin{displaymath}
\xi\blact (a\blact \partial^+_{\mathsf X}(f))=\partial^+_{\mathsf X} ( \xi\blact (a\blact f)),\qquad f\in \Hom_{\kk}((K_H)_r,\mathsf X),\,\xi\in H_1^*,\,a\in A_0.
\end{displaymath}
Let~$u=u_1\tensor u_2\tensor u_3\in (K_H)_{r+1}$, $r\ge 1$, where $u_1, u_3\in H_1$ and $u_2\in (K_H)_{r-1}$. Then 
\begin{align*}
(\xi\blact (a\blact \partial^+_{\mathsf X}(f)))(u)&=\xi(u_3) a_{(1)}\partial^+_{\mathsf X}(f)(u_1\ract a_{(2)}
\tensor u_2\ract a_{(3)})\\&=(-1)^{r}
\xi(u_3)a_{(1)}(u_1\ract a_{(2)})f(u_2\ract a_{(3)})\\
&=(-1)^{r} \xi(u_3) u_1 a_{(1)} f(u_2\ract a_{(2)})\\
&=
(-1)^{r} u_1 (\xi\blact (a\blact f))(u_2\tensor u_3)\\
&=\partial^+_{\mathsf X} ( \xi\blact (a\blact f))(u).
\end{align*}
The claim follows.
\end{proof}

\subsection{}\label{KDS130}
We have the following isomorphisms of $\kk$-vector spaces 
\begin{displaymath}
(\soc \mathsf I)^\bullet(\mathsf X)=\bigoplus_{r\in\ZZ}\Hom_\kk((K_H)_r,\mathsf X)\cong \bigoplus_{r\in\ZZ}(K_H)_r^*\tensor_\kk \mathsf X\cong H^!\tensor_\kk \mathsf X,
\end{displaymath}
where we used~\cite{BGS}*{\S 2.8} for the last identification. The canonical isomorphism 
\begin{displaymath}
\Theta:(K_H)^*_r\tensor_\kk \mathsf X\to \Hom_\kk((K_H)_r,\mathsf X)
\end{displaymath}
is given by
$\xi\tensor x\mapsto \Theta_{\xi\tensor x}$, where we have $\Theta_{\xi\tensor x}(u)=\xi(u)x$, $x\in\mathsf X$, 
$\xi\in (K_H)_r^*$ and $u\in (K_H)_r$.

Following \cite{BGS}*{Section~2.12} and \cite{MOS}*{Section~2.4},
denote by $\mathcal K^{\uparrow}(\Inj(\Gr_{A}^-\cat C))$ the full subcategory of 
$\mathcal K^+(\Inj(\Gr_{A}^-\cat C))$ which consists of all complexes $(\mathsf X^{\bullet},d)$ 
satisfying ${\mathsf X^j}_i=0$ for all $j+i\gg 0$. Similalrly, denote by 
$\mathcal K^{\downarrow}(\Proj(\Gr_{A^\circledast}^+\cat C))$ the full subcategory of 
$\mathcal K^-(\Proj(\Gr_{A^\circledast}^+\cat C))$ which consists of all complexes 
$(\mathsf Y^{\bullet},\partial)$  satisfying ${\mathsf Y^j}_i=0$ for all $j+i\ll 0$.
We note that our notation is consistent with that in
\cite{MOS}*{Section~2.4} and is opposite to that in \cite{BGS}*{Section~2.12}.

\propanch{KDS130}
\begin{prop}
For any $\mathsf X\in\Gr_A^-\cat C$, the complex $(\soc \mathsf I)^\bullet(\mathsf X)$ identifies with the complex 
$(Y^\bullet,\partial^+_{\mathsf X})\in\mathcal K^{\downarrow}(\Proj(\Gr_{A^\circledast}^+\cat C))$, where 
\begin{displaymath}
\mathsf Y^j=(A^{\circledast}\tensor_{A_0}\Pi_j(\mathsf X))\ds{-j}.
\end{displaymath}
In particular, $(\soc \mathsf I)^\bullet$ extends to a functor 
\begin{displaymath}
(\dsoc^+ \mathsf I)^\bullet:\mathcal 
K^{\uparrow}(\Inj(\Gr_{A}^-\cat C))\to \mathcal K^{\downarrow}(\Proj(\Gr_{A^\circledast}^+\cat C)).
\end{displaymath}
\end{prop}

\begin{proof}
The isomorphisms of $\kk$-vector spaces
\begin{displaymath}
A^{\circledast}\tensor_{A_0}\mathsf X\cong H^!\tensor_\kk A_0\tensor_{A_0}\mathsf X\cong H^!\tensor_\kk \mathsf X
\end{displaymath}
induce a left $A^{\circledast}$-module structure on~$H^!\tensor_\kk \mathsf X$ 
(cf.~\subsref{KDS10}) which is given by
\begin{displaymath}
a\lact (\xi\tensor x)=(a_{(2)}\lact \xi)\tensor a_{(1)}\lact x,\qquad \xi'\lact (\xi\tensor x)=\xi'\xi\tensor x,
\end{displaymath}
for all $a\in A_0$, $\xi,\xi'\in H^!$ and~$x\in\mathsf X$. Using this and the identification 
$(K_H)_r^*\cong H^!_r$ we obtain 
\begin{align*}
\Theta_{a\lact (\xi\tensor x)}(u)&=\Theta_{a_{(2)}\lact \xi\tensor a_{(1)}\lact x}(u)=
(a_{(2)}\lact \xi)(u)a_{(1)}\lact x \\
&=\xi(u\ract a_{(2)})a_{(1)}\lact x=(a\blact \Theta_{\xi\tensor x})(u),\\
\intertext{while}
\Theta_{\xi'\lact(\xi\tensor x)}(u)&=\Theta_{\xi'\xi\tensor x}(u)=\xi(u_1)\xi'(u_2)x=
(\xi'\blact \Theta_{\xi\tensor x})(u)
\end{align*}
for all $a\in A_0$, $\xi,\xi'\in H^!$, $u=u_1\tensor u_2\in K_H$ and~$x\in\mathsf X$.

Thus, we have $(\soc I)^\bullet(\mathsf X)\cong \bigoplus_{r,j\in\ZZ} A^\circledast_r\tensor_{A_0} \mathsf X_j$ 
as an $A^\circledast$-module. Set  
\begin{displaymath}
\mathsf Y^j=(A^\circledast\tensor_{A_0} \Pi_j(\mathsf X))\ds{-j}.
\end{displaymath}
Then $(\soc I)^\bullet(\mathsf X)\cong \bigoplus_{j\in \ZZ} \mathsf Y^j$ and 
$\mathsf Y^j{}_r=A^\circledast{}_{r+j}\tensor_{A_0} \Pi_j(\mathsf X)$. Clearly, 
$\mathsf Y^j\in\Proj(\Gr_{A^\circledast}^+\cat C)$. Since $\partial(A^\circledast{}_{r+j}\tensor_{A_0} 
\Pi_j(\mathsf X))\subset A^{\circledast}_{r+j+1}\tensor_{A_0}\Pi_{j+1}(\mathsf X)$ under our identifications,
it follows that $\partial^+_{\mathsf X}:\mathsf Y^j\to\mathsf Y^{j+1}$ is a morphism 
in~$\Gr_{A^\circledast}^+\cat C$. Thus, $(\mathsf Y^\bullet,\partial^+_{\mathsf X})$ is a 
complex in~$\Proj(\Gr_{A^\circledast}^+\cat C)$. Since $\mathsf X_j=0$ if 
$j\gg 0$, we conclude that $(\mathsf Y^\bullet,\partial)$ is bounded above. It is easy to see that 
we even have $(\mathsf Y^\bullet,\partial)\in \mathcal K^{\downarrow}(\Proj(\Gr_{A^\circledast}^+\cat C))$.

It remains to note that, similarly to the classical Koszul duality, the natural extension 
$(\dsoc^+ \mathsf I)^\bullet$ of $(\soc \mathsf I)^\bullet$ maps
$\mathcal K^{\uparrow}(\Inj(\Gr_{A}^-\cat C))$ to $\mathcal K^{\downarrow}(\Proj(\Gr_{A^\circledast}^+\cat C))$.
\end{proof}

\subsection{}\label{KDS135}
Now we consider the dual picture. Retain the assumptions of~\subsref{KDS50}. We will
decorate all functors related to $\Gr^\pm_{A^\circledast}\cat C$ with $\circledast$ to distinguish them 
from those related to $\Gr^\pm_{A}\cat C$. Let $\mathsf Y$ be an object in~$\Gr^-_{A^\circledast} \cat C$.
Then 
\begin{displaymath}
\mathsf P^\circledast{}^{-r} (\mathsf Y)=A^\circledast\tensor_{A_0} (K_{A^\circledast})_r\tensor_{A_0} \mathsf Y
\cong H^!\tensor_\kk A_0\tensor_{A_0} (K_{H^!})_r\tensor_\kk A_0\tensor_{A_0}\mathsf Y
\cong H^!\tensor_\kk (K_{H^!})_r\tensor_\kk \mathsf Y,
\end{displaymath}
hence 
\begin{displaymath}
(\Top\mathsf P^\circledast )^{-r}(\mathsf Y)\cong(K_{H^!})_r\tensor_\kk\mathsf Y.
\end{displaymath}
Since~$(H^!)^!\cong H$, we have that $(K_{H^!})_r$ identifies with~$H^*_r$ and so 
\begin{displaymath}
(\Top\mathsf P^\circledast )^{-r}(\mathsf Y)=H^*_r\tensor_\kk\mathsf Y\cong \Hom_\kk(H_r,\mathsf Y)
\end{displaymath}
where the isomorphism is given by $\xi\tensor y\mapsto (h\mapsto \xi(h)y)$. Thus, 
$(\Top\mathsf P^\circledast )^\bullet(\mathsf Y)$ identifies with $\mathsf I^0(\mathsf Y)$ where we regard 
$\mathsf Y$ as an $A$-module via the canonical epimorphism $A\twoheadrightarrow A_0$.

\propanch{KDS135}
\begin{prop}
For any $\mathsf Y\in\Gr_{A^\circledast}^+\cat C$, we have that $(\Top P^\circledast )^\bullet(\mathsf Y)$ is 
isomorphic to~$\mathsf I^0(\mathsf Y)$ as a left $A$-module and the $A$ action commutes with 
the differential~$\partial^\circledast{}^-_{\mathsf Y}$. In particular, 
$(\Top P^\circledast{})^\bullet(\mathsf Y)$ identifies with 
$(\mathsf X^\bullet,\partial^\circledast{}^-_{\mathsf Y})\in \mathcal K^{\uparrow}(\Inj(\Gr_A^-\cat C))$,
where $\mathsf X^j=\mathsf I^0(\Pi_j(\mathsf Y))\ds{j}$ and so $(\Top P^\circledast{})^\bullet$ 
extends to a functor 
\begin{displaymath}
(\dtop^- P^\circledast{})^\bullet:\mathcal K^{\downarrow}(\Proj(\Gr_{A^\circledast}^+ \cat C))
\to \mathcal K^{\uparrow}(\Inj(\Gr_A^- \cat C)). 
\end{displaymath}
\end{prop}

\begin{proof}
We have 
\begin{displaymath}
\mathsf I^0(\mathsf Y)\cong \bigoplus_{i\in\ZZ} \Hom_\kk(H_i,\mathsf Y)\cong (\Top\mathsf P^\circledast )^\bullet(\mathsf Y),
\end{displaymath}
so we only need to check that the $A$-action on $\mathsf I^0(\mathsf Y)$ commutes with the differential 
induced by $d$. For this we observe that the action of $H$ on~$(K_{H^!})_r$ which yields the 
$A$ action is induced by the action of $H_1$ by {\em left} contractions, while the differential is 
given by the action of the rightmost factor of the tensor product on~$\mathsf X$ (cf.~\subsref{KDS70}).
The rest of the argument repeats the one in the proof of~\propref{KDS130} and is omitted.
\end{proof}

\subsection{}\label{KDS150}
Now we can establish the main property of functors constructed above in \propref{KDS130} and~\propref{KDS135}.

\propanch{KDS150}
\begin{prop}
Let $\cat C$ be a category of left $A_0$-modules satisfying \eqref{cnd:C1}--\eqref{cnd:C4} for 
$A=A_0\ltimes H$. Then
\begin{enumerate}[{\rm(a)}]
\item\label{prop:MA150.a} $(\dtop^-\mathsf P^\circledast)^\bullet\circ 
(\dsoc^+\mathsf I)^\bullet\cong\mathsf I^\bullet$ 
as endofunctors on $\mathcal K^{\uparrow}(\Inj(\Gr_A^-\cat C))$.
\item\label{prop:MA150.b} $(\dsoc^+\mathsf I)^\bullet\circ(\dtop^-\mathsf P^\circledast)^\bullet 
\cong\mathsf P^\circledast{}^\bullet$ as endofunctors on $\mathcal K^{\downarrow}(\Proj(\Gr_{A^\circledast}^+\cat C))$.
\end{enumerate}
\end{prop}

\begin{proof}
Let $\cat X$ denote the smallest triangulated subcategory of $\mathcal K^{\uparrow}(\Inj(\Gr_A^-\cat C))$
containing $\cat C$ (which is identified with a subcategory of~$\Gr_A^-\cat C$ via the full exact embedding $\gr_0$, see~\subsref{GEN.5}) and closed under the degree shift. Then every object in
$\mathcal K^{\uparrow}(\Inj(\Gr_A^-\cat C))$ can be considered as a limit of objects
in $\cat X$ in the usual way (by cutting increasing finite pieces both in position and
gradings). Similarly, every object in $\mathcal K^{\downarrow}(\Proj(\Gr_{A^\circledast}^+\cat C))$
is a limit of objects in the category $\cat Y$, where the latter is 
defined as the smallest triangulated subcategory 
of $\mathcal K^{\downarrow}(\Proj(\Gr_{A^\circledast}^+\cat C))$
containing $\cat C$ and closed under the degree shift. 
Therefore it suffices to prove both assertions for 
an object $X$ in~$\mathscr C$. 

Using~\propref{KDS130}, we conclude that $(\dsoc^+\mathsf I)^\bullet(X)$ identifies with the 
trivial complex of $\mathsf P^\circledast{}^0(X)$ (cf.~\subsref{GEN.new1}). Furthermore,
\begin{displaymath}
(\dtop^-\mathsf P^\circledast)^\bullet\circ (\dsoc^+\mathsf I)^\bullet(X)=
(\Top P^\circledast)^\bullet(\mathsf P^\circledast{}^0(X))= (\mathsf Z^\bullet,d),
\end{displaymath}
where, by \subsref{KDS130}--\subsref{KDS150}, we have
\begin{multline*}
\mathsf Z^j=\mathsf I^0(\Pi_j(\mathsf P^\circledast{}^0(X)))\ds{j}=\mathsf I^0(A^\circledast_j\tensor_{A_0}X)\ds{j}
=\bigoplus_{i\in\ZZ}\Hom_{A_0}(A_i,A^{\circledast}_j\tensor_{A_0} X)\ds{j}\\
\cong \bigoplus_{i\in\ZZ} \Hom_{\kk}(H_i,H^!_j\tensor_\kk X)\ds{j}
\cong \bigoplus_{i\in\ZZ} \Hom_\kk((H^!)_j^*\tensor_\kk H_i,X)\\
\cong \bigoplus_{i\in\ZZ} \Hom_\kk( (K_H)_j\tensor_\kk H_i,X)
\cong \mathsf I^j(X)
\end{multline*}
as objects in $\Gr_A^-\cat C$. Furthermore, note that on~$\mathsf Z^\bullet$ the differential 
$d_{-r}$ is given by the map $\bar m^!_{r,i}\tensor \id_X$, where
$\bar m^!_{r,i}: (K_{H^!})_r\tensor_\kk H^!{}_i\to  (K_{H^!})_{r-1}\tensor_\kk H^!{}_{i+1}$ 
was defined in~\subsref{KDS50}. Since $\Hom_\kk( (K_H)_j\tensor_\kk H_i,X)\cong ((K_H)_j
\tensor_\kk H_i)^*\tensor_\kk X$ and, under that identification, $(\partial^+_X f)\tensor x=f\circ 
\bar m_{j+1,i-1}\tensor x$, $f\in ( (K_H)_j\tensor_\kk H_i)^*$, it follows from~\lemref{KDS50}
that $(\mathsf I^\bullet(X),\partial^+_X)$ is naturally isomorphic to $(\mathsf Z^\bullet,d)$. This proves~\eqref{prop:MA150.a}.

Similarly, we identify $(\dtop^-\mathsf P^\circledast)^\bullet(X)$ with the trivial complex of~$\mathsf I^0(X)$. Then
\begin{displaymath}
(\dsoc^+\mathsf I)^\bullet\circ (\dtop^-\mathsf P^\circledast)^\bullet(X)
=(\soc \mathsf I)^\bullet(\mathsf I^0(X))=(\mathsf Y^\bullet,\partial),
\end{displaymath}
where 
\begin{multline*}
\mathsf Y^j=(A^{\circledast}\tensor_{A_0}\Pi_j(\mathsf I^0(X)))\ds{-j}
\cong H^!\tensor_\kk \Hom_\kk(H_j,X)\ds{-j}\\
\cong H^!\tensor_\kk (H_j)^*\tensor_\kk X=\mathsf P^\circledast{}^{-j}(X).
\end{multline*}
It is now easy to see that the differential~$\partial$ coincides with $\partial^\circledast{}^-_{X}$.
\end{proof}

\subsection{}\label{KDS170}
The following theorem is the main result of this paper and provides 
an analogue of Koszul duality for our setting.

\thmanch{KDS170}
\begin{thm}
Let $A_0$ be a $\kk$-bialgebra. Let $H$ be a right $A_0$-module algebra 
with $H_0=\kk$, which is quadratic with $H_1$ finite dimensional. Let $A=A_0\ltimes H$, $A^\circledast=H^!\rtimes A_0^{\cop}$ and
let $\cat C$ be a category of left $A_0$-modules satisfying \eqref{cnd:C1}--\eqref{cnd:C4}.
Then the following are equivalent:
\begin{enumerate}[{\rm(a)}]
\item\label{KDS170.a} $(\dtop^-\mathsf P^\circledast)^\bullet\circ (\dsoc^+\mathsf I)^\bullet$ is 
isomorphic to the identity functor on $\mathcal K^{\uparrow}(\Inj(\Gr_A^-\cat C))$.
\item\label{KDS170.a'} $(\dsoc^+\mathsf I)^\bullet\circ(\dtop^-\mathsf P^\circledast)^\bullet $ is 
isomorphic to the identity functor on $\mathcal K^{\downarrow}(\Proj(\Gr_{A^\circledast}^+\cat C))$.
\item\label{KDS170.b} For all $X$ in~$\mathscr C$, $X\cong \mathsf I^\bullet(X)$ in
$\mathcal D^+(\underline{\Inj(\Gr_A^-\cat C)})$.
\item\label{KDS170.b'} For all $X$ in~$\mathscr C$, $X\cong \mathsf P^\circledast{}^\bullet(X)$ in 
$\mathcal D^-(\overline{\Proj(\Gr_{A^\circledast}^+\cat C)})$.
\item\label{KDS170.c} $H$ is Koszul.
\end{enumerate}
\end{thm}

In view of this theorem, it is natural to call either of the dual properties~\eqref{KDS170.b} 
and~\eqref{KDS170.b'} the $\mathscr C$-Koszulity. Note that, due to Lemma~\ref{GEN.85}, we always have
for $X$ in~$\mathscr C$
\begin{equation}\label{eq:K-arrow}
\mathsf I^\bullet(X)\in \mathcal K^{\uparrow}(\Inj(\Gr_A^-\cat C))\quad \text{ and }\quad
\mathsf P^\circledast{}^\bullet(X)\in \mathcal K^{\downarrow}(\Proj(\Gr_{A^\circledast}^+\cat C)). 
\end{equation}

\begin{proof}
By~\propref{KDS150}, \eqref{KDS170.a} is equivalent to~\eqref{KDS170.b} and \eqref{KDS170.a'} is 
equivalent to~\eqref{KDS170.b'}. Furthermore, identifying~$X$ with its trivial complex (cf.~\subsref{GEN.new1}) we have
$\mathsf H_0(X)=X\cong \mathsf H_0(\mathsf I^\bullet(X))$ by \lemref{GEN.85}\eqref{lem:GEN.85.0}. Then
by \lemref{GEN.85}\eqref{lem:GEN.85.b}, part \eqref{KDS170.b} is equivalent to the exactness of the 
complex $\mathsf I^\bullet(X)$. By~\eqref{cnd:C2}--\eqref{cnd:C1}, $\mathsf I^\bullet(X)$ is exact 
if and only if the right Koszul complex $K_H\tensor_\kk H$ is exact, which by~\cite{BGS}*{Theorem~2.6.1} 
is equivalent to~$H$ being Koszul. Thus, \eqref{KDS170.b}$\iff$\eqref{KDS170.c}. Similarly, \eqref{KDS170.b'} 
is equivalent to the exactness of~$\mathsf P^\circledast{}^\bullet(X)$ which is equivalent to the exactness 
of the left Koszul complex for~$H^!$ which happens if and only if~$H^!$ is Koszul. By \cite{BGS}*{Proposition~2.9.1},
if~$H$ is Koszul then so is~$H^!$.  Since in our case $(H^!)^!\cong H$, it follows that $H^!$ is Koszul if 
and only if~$H$ is Koszul. Thus, \eqref{KDS170.b'}$\iff$\eqref{KDS170.c}.
\end{proof}

Note that, in the case $A_0$ is semi-simple, the above theorem reduces to the classical Koszul duality
from \cite{BGS}. The other extreme case $H=H_0=\Bbbk$ provides a trivial generalization of the
classical Koszul duality for a semi-simple algebra (concentrated in the zero component) to the case of a 
non semi-simple algebra concentrated in the zero component.

\subsection{}\label{KDS190}
As an immediate corollary of~\thmref{KDS170} we obtain the following analogue of 
\cite{BGS}*{Theorem~2.12.5}, \cite{MOS}*{Theorem~30}  and \cite{Ma}*{Theorem~4.3.1} for our setting.
Define $\mathcal D^{\uparrow}(\underline{\Inj(\Gr_A^-\cat C)})$  as the 
isomorphim closure of $\mathcal K^{\uparrow}(\Inj(\Gr_A^-\cat C))$ inside the category
$\mathcal D^+(\underline{\Inj(\Gr_A^-\cat C)})$. Similarly, we define
$\mathcal D^{\downarrow}(\overline{\Proj(\Gr_{A^\circledast}^+\cat C)})$
as the isomorphim closure of $\mathcal K^{\downarrow}(\Proj(\Gr_{A^\circledast}^+\cat C))$
inside $\mathcal D^-(\overline{\Proj(\Gr_{A^\circledast}^+\cat C)})$. By definition,
we have equivalences
\begin{gather*}
\mathcal D^{\uparrow}(\underline{\Inj(\Gr_A^-\cat C)})\cong
\mathcal K^{\uparrow}(\Inj(\Gr_A^-\cat C))\\
\mathcal D^{\downarrow}(\overline{\Proj(\Gr_{A^\circledast}^+\cat C)})\cong
\mathcal K^{\downarrow}(\Proj(\Gr_{A^\circledast}^+\cat C)) .
\end{gather*}
Note that if $\mathscr C$ is semisimple then $\Gr_A^-\cat C$ and $\Gr_{A^\circledast}^+\cat C$ are abelian and 
$\mathcal D^{\uparrow}(\Gr_A^-\cat C)$, $\mathcal D^{\downarrow}(\Gr_{A^\circledast}^+\cat C)$ can be defined as in~\cite{BGS}*{\S2.12} and~\cite{MOS}*{\S2.4}.
\coranch{KDS190}
\begin{cor}
Suppose that $H$ is Koszul. Then:
\begin{enumerate}[{\rm(a)}]
 \item the categories 
$\mathcal D^{\uparrow}(\underline{\Inj(\Gr_A^-\cat C)})$ and 
$\mathcal D^{\downarrow}(\overline{\Proj(\Gr_{A^\circledast}^+\cat C)})$ are equivalent;
\item if $\cat C$ is semisimple then $\mathcal D^{\uparrow}(\Gr_A^-\cat C)\cong \mathcal D^{\downarrow}(\Gr_{A^{\circledast}}^+\cat C)$.
\end{enumerate}
\end{cor}
\begin{proof}
The first assertion is immediate from the definitions and \thmref{KDS170}. The second assertion follows from \thmref{KDS170} and~\eqref{eq:K-arrow}.
\end{proof}

\section{Examples}\label{snew5}

\subsection{}
Let $\lie g$ be a semi-simple Lie algebra and $V$ a finite dimensional $\lie g$-module.
Consider the generalized Takiff Lie algebra $\lie g\ltimes V$ as in \subsref{KDS20}.
Take~$\widehat{\cat C}=\mathscr C$ to be the (abelian) category of finite dimensional $\lie g$-modules.
Clearly, $\cat C$ satisfies the conditions \eqref{cnd:C1}--\eqref{cnd:C4}. This
is exactly the case considered in \cite{CG-2} (where $V$ was the adjoint representation of~$\lie g$) and then in~\cite{CKR}.

Let $H=S(V)$ which we regard as a {\em right} $U(\lie g)$-module algebra. Since $U(\lie g)$ 
is cocommutative and $H^!=\twedge V^*$ as a {\em left} $U(\lie g)$-module algebra, our construction 
establishes a duality between a category of graded modules with graded component in~$\cat C$ over the 
generalized Takiff Lie algebra $\lie g\ltimes V$ and a category of graded modules with graded components 
in~$\cat C$ over the Lie superalgebra $\lie s=\lie s_0\oplus\lie s_1$ with $\lie s_0=\lie g$ and $\lie s_1= V$, 
see \subsref{KDS20} for the definition. Note that we should take $V$ and not~$V^*$ as a 
$\lie g$-module in the definition of~$\lie s$.

\subsection{}
The previous example admits a non-trivial generalization. For the same  $\lie g$, $V$  and $\lie g\ltimes V$ 
as in the previous example, choose as $\widehat{\cat C}$ the Bernstein-Gelfand-Gelfand category 
$\mathscr O_{\lie g}$ (\cite{BGG}). Then $\mathscr C$ can be taken to be one of the following categories
\begin{itemize}
\item[-] $\Proj(\cat O_{\lie g})$;
\item[-] $\Inj(\cat O_{\lie g})$;
\item[-] the category of tilting modules in~$\cat O_{\lie g}$ (see~\cite{CR} or \cite{Hu}*{Chapter~11}).
\end{itemize}
Conditions~\eqref{cnd:C1} and~\eqref{cnd:C2} are clear for all these subcategories while conditions
\eqref{cnd:C3} and~\eqref{cnd:C4} follow from the observation that all of them are closed under 
tensoring with finite dimensional $\lie g$-modules. 

Our results again establish Koszul duality between the category of graded modules with graded component 
in~$\cat C$ over $\lie g\ltimes V$ and the  category of graded modules with graded components 
in~$\cat C$ over the corresponding generalized Takiff Lie superalgebra from \subsref{KDS20}.

\subsection{}
Another family of examples can be obtained in the framework proposed in~\cites{BZ,Z1}. Let $A_0$ be the quantized enveloping algebra~$U_q(\lie g)$
and let $H=S_q(V_q)$ where $V_q$ is a right $U_q(\lie g)$-module and $S_q(V_q)$ is its braided symmetric algebra as defined in~\cite{BZ}.
If $S_q(V_q)$ is a flat deformation of the symmetric algebra $S(V)$ of the $q=1$ limit~$V$ of~$V_q$, then it is Koszul and its quadratic 
dual is the braided exterior algebra of~$V_q{}^*$ which again is regarded as a left module. 
Our construction thus establishes a duality between a category of graded modules with finite dimensional
graded pieces over the semidirect products of $U_q(\lie g)$ with the braided symmetric algebra and
with the braided exterior algebra, respectively. It should be noted that $\lie g$-modules~$V$ for which $S_q(V_q)$ is a flat deformation
of~$S(V)$ are rather rare.
All simple modules with this property where classified in~\cite{Z}, and it should be noted that 
the adjoint representation of~$\lie g$ is not among them. The first non-simple example of~$V$ for which 
$S_q(V_q)$ is flat was constructed in~\cite{BG}
where the corresponding semidirect product is one of the key ingredients of the construction. The relationship between generalized 
Takiff algebras and semidirect products $U_q(\lie g)\ltimes S_q(V_q)$ was studied in~\cite{Z1}.

\end{document}